\documentclass[11pt]{amsart}

\usepackage{amssymb}
\usepackage[all]{xy}
\usepackage[colorlinks=true, urlcolor=rltblue, citecolor=drkgreen, linkcolor=drkred] {hyperref}
\usepackage{color}
\usepackage{pdfsync}
\definecolor{rltblue}{rgb}{0,0,0.4}
\definecolor{drkred}{rgb}{0.6,0,0}
\definecolor{drkgreen}{rgb}{0,0.4,0}
\usepackage{graphicx}
\usepackage[mathscr]{eucal}

\usepackage{lmodern}
\usepackage[capitalize]{cleveref}
\usepackage{amssymb,amsmath,amsthm}
\usepackage{mathtools, MnSymbol}
\usepackage{setspace}
\usepackage{tikz-cd}
\usetikzlibrary{decorations.pathreplacing}
\usepackage[T1]{fontenc}
\usepackage[utf8]{inputenc}
\usepackage{textcomp} 
\usepackage{stmaryrd}
\usepackage{datetime}
\usepackage{todonotes}
\usepackage{xcolor}
\usepackage{mdframed}
\usepackage[all]{xy}

\usepackage{thmtools}
\usepackage{enumitem}

\declaretheorem[numberwithin=section]{theorem}
\declaretheorem[sibling=theorem]{lemma}
\declaretheorem[sibling=theorem]{proposition}
\declaretheorem[sibling=theorem]{corollary}
\declaretheorem[sibling=theorem]{definition}
\declaretheorem[sibling=theorem]{question}
\declaretheorem[sibling=theorem]{example}

\newcommand{\SR}{\text{SR}}

\newcommand{\Pinf}[1]{\Pi^{\mathrm{in}}_{#1}}
\newcommand{\Sinf}[1]{\Sigma^{\mathrm{in}}_{#1}}

\newcommand{\bigwwedge}{%
  \mathop{
    \mathchoice{\bigwedge\mkern-15mu\bigwedge}
               {\bigwedge\mkern-12.5mu\bigwedge}
               {\bigwedge\mkern-12.5mu\bigwedge}
               {\bigwedge\mkern-11mu\bigwedge}
    }
}

\newcommand{\bigvvee}{%
  \mathop{
    \mathchoice{\bigvee\mkern-15mu\bigvee}
               {\bigvee\mkern-12.5mu\bigvee}
               {\bigvee\mkern-12.5mu\bigvee}
               {\bigvee\mkern-11mu\bigvee}
    }
}

\def\sA{\mathcal A}
\def\sB{\mathcal{B}}
\def\sC{\mathcal{C}}
\def\sD{{\mathcal D}}
\def\A{\mathcal A}
\def\B{\mathcal{B}}

\def\E{\mathcal{E}}

\def\K{\mathcal K}
\def\L{\mathcal L}
\def\M{{\mathcal M}}

\def\N{{\mathcal N}}

\def\bQ{{\mathbb Q}}

\title{Generically Computable Linear Orderings}

\author{Wesley Calvert}
\address{Department of Mathematics\\ Mail Code 4408\\
Southern Illinois University, Carbondale\\
1245 Lincoln Drive\\
Carbondale, Illinois 62901}
\email{wcalvert@siu.edu}
\author{Douglas Cenzer}
\address{Department of Mathematics, University of Florida,
Gainesville, FL 32611}
\email{cenzer@ufl.edu}
\author{David Gonzalez}
\address{Department of Mathematics, University of California Berkeley, Berkeley, CA 94720}
\email{david\_gonzalez@berkeley.edu}
\author{Valentina Harizanov}
\address{Department of Mathematics, George Washington
University, Washington, DC 20052}
\email{harizanv@gwu.edu}

\thanks{Harizanov was partially supported by FRG NSF grant DMS-2152095.
Gonzalez was supported by an ASL Student Travel Grant to attend the 2023 European Summer Meeting of the ASL where the first ideas for this paper were conceived.}

\date{\today}

\begin{document}

\maketitle

\begin{abstract}
    We study notions of generic and coarse computability in the context of computable structure theory.
    Our notions are stratified by the $\Sigma_\beta$ hierarchy.
    We focus on linear orderings.
    We show that at the $\Sigma_1$ level all linear orderings have both generically and coarsely computable copies.
    This behavior changes abruptly at higher levels; we show that at the $\Sigma_{\alpha+2}$ level for any $\alpha\in\omega_1^{ck}$ the set of linear orderings with generically or coarsely computable copies is $\mathbf{\Sigma}_1^1$-complete and therefore maximally complicated.
    This development is new even in the general analysis of generic and coarse computability of countable structures.
    In the process of proving these results we introduce new tools for understanding generically and coarsely computable structures.
    We are able to give a purely structural statement that is equivalent to having a generically computable copy and show that every relational structure with only finitely many relations has coarsely and generically computable copies at the lowest level of the hierarchy.
\end{abstract}

\section{Introduction}


Many results in computable structure theory seem to rely on the construction of a rare pathological example.  Considering computable structure theory from the perspective of generic or coarse computability is a way to assess which results truly rely on such exceptional instances, and which reflect the more ``normal'' instances.  In the present paper, we apply this approach to linear orderings.

Linear orderings are a well-studied benchmark example in computable structure theory, and have frequently served as the impetus for broader research programs.  The literature is vast, but we mention in particular the pivotal role of Church and Kleene's work on computable ordinals \cite{ChurchKleene1937}; Harrison's study of computable linear orderings with no computable descending sequence \cite{Harrison1968}; the work of Remmel \cite{Remmel1981}, Goncharov, and Dzgoev \cite{GoncharovDzgoev1980} on computably categorical linear orderings; the work of Richter \cite{Richter1981} and Knight \cite{Knight1986} on degrees coded in jumps of orderings; and, in particular, Jockusch and Soare's construction of a low linear ordering not isomorphic to a computable linear ordering \cite{JockuschSoare1991}.  Each of these results told us not only about linear orderings, but gave us important ingredients for understanding computable structures in general.  A survey of early work can be found in \cite{Downey1998}.

The notion of dense computability for sets of natural numbers is that there is an algorithm which computes the solution on an  asymptotically dense set. The study of densely computable, generically computable, and coarsely computable sets is now well-established. 

The classic motivating example, which comes from computable structure theory, is the word problem  for finitely generated groups.  The Novikov-Boone Theorem \cite{Nov58,Boone58} showed that there exist finitely presented groups with undecidable word problem. For many groups with undecidable word problems, including a standard example from \cite{Rotman},  the particular words on which it is difficult to decide equality to the identity are very special words (and are even called by this term in some expositions).  Thus the problem can be solved on a dense set.
Kapovich et al. \cite{KMSS03} defined \emph{generic-case} case complexity and showed that for a large class of finitely presented groups, the word problem has linear-time generic complexity. 

Jockusch and Schupp \cite{JS12} generalized the density approach beyond word problems on groups to the broader context of computability theory in the following way.
\begin{definition}
  Let $S \subseteq \mathbb{N}$.
  \begin{enumerate}
  \item The density of $S$ up to $n$, denoted by $\rho_n(S)$, is given by \[\rho_n(S) = \frac{\left|S \cap \{0, 1, 2, \dots, n\}\right|}{n+1}.\]
  \item The asymptotic density of $S$, denoted by $\rho(S)$, is given by $\lim\limits_{n \to \infty}\rho_n(S)$.
  \end{enumerate}
\end{definition}

A set $A$ is said to be \emph{generically computable} if and only if there is a partial computable function $\phi$ such that $\phi$ agrees with $\chi_A$ throughout the domain of $\phi$, and such that the domain of 
$\phi$ has asymptotic density 1.  A set $A$ is said to be \emph{coarsely computable} if and only if there is a \emph{total} computable function $\phi$ that agrees with $\chi_A$ on a set of asymptotic density 1.

In three recent papers \cite{CCH22,CCH23,CCHGroups}, the first, second, and fourth authors have developed the notions of \emph{densely computable} structures and isomorphisms. This builds on the concepts of generically and coarsely computable sets, as studied by Jockusch and Schupp \cite{JS12,JS17} and many others, which have been a focus of research in computability. For structures, the question is whether some ``large'' substructure is computable.

Here are some of the basic definitions needed for this paper. 

\begin{definition} Let $S \subseteq \omega$.
 We say that $S$ is \emph{generically computable} if there  is a partial computable function $\Phi: \omega \to 2$ such that
  $\Phi = \chi_S$ on the domain of $\Phi$, and such that the domain of
  $\Phi$ has asymptotic density 1.
\end{definition}

The most natural notion for a structure seems to be the requirement  that the substructure with domain $D$ resembles the given structure $\sA$ by agreeing on certain formulas, existential infinitary formulas in particular. 
We recall the notion of an elementary substructure. 

\begin{definition} A substructure $\sB$ of the structure $\sA$ is said to be an
\emph{elementary} substructure ($\sB \prec \sA$) if for any $b_1,\dots,b_n \in \sB$, 
and any first order formula $\phi$, we have
\[\sA \models \phi(b_1,\dots,b_n) \iff \sA \models \phi(b_1,\dots,b_n).\]

The substructure $\sB$ is said to be a $\Sigma_\alpha$ elementary substructure  ($\sB \prec_\alpha \sA$) if for any $b_1,\dots,b_n \in \sB$, 
and any $\Sigma_\alpha$  formula $\phi$, we have 
\[\sA \models \phi(b_1,\dots,b_n) \iff \sA \models \phi(b_1,\dots,b_n).\]
\end{definition}

We note that there is some question in the previous definition whether it is more appropriate to use only the first-order $\Sigma_n$ formulas, or infinitary $\Sigma_n$ formulas for the finite levels of the hierarchy.  In the results of the present paper, we frequently have the best of both worlds --- We frequently prove the existence of a structure that is $\Sigma_n$ elementary in the sense of satisfying the same infinitary $\Sigma_n$ formulas, but often only use from the hypothesis of elementarity that first-order $\Sigma_n$ formulas are preserved. That being said, we will assume we mean infinitary formulas as a default. Infinitary formulas allow us to move into ordinal levels of the hierarchy that are unseen to first-order formulas but still have interesting behavior in the cases we examine.

Throughout the paper, we will assume that all structures, unless otherwise specified, are countable and have universe $\omega$.

\begin{definition} \label{def:genrel} For any structure $\sA$:   
\begin{enumerate}

\item $\sA$ is  \emph{generically computable} if there is a substructure $\sD$ with universe a c.e.\ set $D$ of asymptotic density one such that the substructure $\sD$ with universe $D$ has that each function and  relation  is computable on $\sD$. That is, for any $k$-ary function $f$ and any $k$-ary  relation $R$, both $f \restriction \sD^k$ and $\chi_R \restriction \sD^k$  are the restrictions to $\sD^k$ of partial computable functions.
\item A substructure $\sB$ of $\sA$, with universe $B$, is a \emph{computably enumerable (c.e.) substructure} if the set $B$ is c.e., each relation is c.e.\ and the graph of each function is c.e.\ (so that the function is partial computable but also total on $B$).  
\item
$\sA$ is \emph{$\Sigma_\alpha$-generically c.e.}  if there is a dense c.e.\ set $D$ such that the substructure $\sD$ with universe $D$ is a c.e.\ substructure and also a 
$\Sigma_\alpha$ elementary substructure of $\sA$.
\end{enumerate}
\end{definition}

For $n>0$,  any $\Sigma_{n+1}$-generically c.e.\ structure is  $\Sigma_n$-generically c.e..  For structures with functions  but no relations, this also holds for $n=0$. However, a c.e.\ substructure might not be computable, so a structure $\sA$  with relations   which is $\Sigma_1$-generically c.e.\ is not necessarily generically computable. 

There are, roughly, two extremal possibilities (say, in the case of generic computability):
\begin{enumerate}
\item Every countable structure has a generically computable copy, or
\item Any countable structure with a generically computable copy has a computable copy.
\end{enumerate}
It was shown in \cite{CCH22}  that each of these can be achieved in certain classes, and that they do not exhaust all possibilities.

The authors of \cite{CCH22,CCH23,CCHGroups} also explored these conditions under the added hypothesis that the ``large'' substructures in question be, in some weak sense, elementary.  Again, we find that there are natural extremal possibilities, and that both they and non-extremal cases are achieved.

The authors of \cite{CCH22,CCH23,CCHGroups} also introduced and studied analogous definitions for coarse computability.

\begin{definition}
    \begin{enumerate}
        \item A structure $\sA$ is \textit{coarsely computable} if there exist a dense subset $D$ of the domain of $\sA$ and a computable structure $\mathcal{E}$ such that the structure $\sD$ with universe $D$ is a substructure of both $\sA$ and of $\mathcal{E}$.
        \item Let $\alpha\in\omega_1$. A structure $\sA$ is $\Sigma_\alpha$-\textit{coarsely c.e.} if there exist a dense set $D$ and a c.e.\ structure $\mathcal{E}$ such that the substructure $\sD$ with universe $D$ is a $\Sigma_\alpha$ elementary substructure of both $\sA$ and $\mathcal{E}$.
    \end{enumerate}
\end{definition}

As the elementarity hypotheses are strengthened, all previously known cases eventually (for $\Sigma_n$ elementarity at sufficiently large $n$) trivialize.  This demonstrates that these notions of dense computability are structural --- they depend fundamentally on the semantics of the structure and not only on the density or algorithmic features of the presentation.

At this point, linear orderings become an important example for densely computable structure theory.  All of the examples treated in \cite{CCH22,CCH23,CCHGroups} have bounded Scott rank, so it is natural to expect the eventual trivialization effect.  Linear orderings, on the other hand, provide examples of Scott rank bounded only by the general properties of computable structures.

Ultimately, this is made possible by the fact that the $\Sigma_\alpha$ elementary substructure relation is non-trivial for every $\alpha$ in the case of linear orderings.
It is illustrative to consider some instances of this at low levels of the hierarchy.
For example, the natural numbers embed into the integers in the classical manner, and this induces an embedding of linear orderings.
That being said, this is not a $\Sigma_1$ elementary substructure.
This is due to the fact that the smallest element of the natural numbers, $0$ has the property that $\mathbb{Z}\models\exists x ~ x<0$ yet $\mathbb{N}\models\lnot\exists x ~ x<0$.
Another classical example would be the embedding of linear orderings given by considering $(0,1)_\mathbb{Q}$ as a natural subset of $[0,1]_\mathbb{Q}$.
It will be easy to show using techniques later described in this article that this is a $\Sigma_1$ elementary substructure.
However, this is not a $\Sigma_2$ elementary substructure.
This can be seen by noting that $(0,1)_\mathbb{Q}\models\lnot\exists x\forall y~y\geq x$ yet $[0,1]_\mathbb{Q}\models\exists x\forall y~y\geq x$.

More generally, Ash showed (\cite{cst2} Lemma II.38) that $\omega^\alpha\equiv_{2\alpha}\omega^\beta$ for $\alpha<\beta<\omega^{ck}_1$.
It follows from this that the unique convex, initial embedding of $\omega^\alpha$ into $\omega^\beta$ represents a $\Sigma_{2\alpha}$ substructure.
However, it is also well known that $\omega^\alpha$ has a $\Pi_{2\alpha+1}$ Scott sentence (\cite{cst2} Lemma II.18).
This means that this is not a $\Sigma_{2\alpha+1}$ substructure as $\omega^\beta$ is not isomorphic to $\omega^\alpha$ and therefore does not satisfy the Scott sentence of $\omega^\alpha$.
These specific examples of embeddings do not involve any non-computable structures.
In particular, they do not give any interesting behavior when it comes to dense computability.
That being said, the unbounded in $\omega_1^{ck}$ nature of non-trivial $\Sigma_\alpha$ substructures is a prerequisite for there to be interesting behavior for $\Sigma_\alpha$ dense computability.
Using more complex constructions, we will be able to show interesting behavior for $\Sigma_\alpha$ dense computability unbounded in $\omega_1^{ck}$, just as we might suspect based on this initial analysis.

The plan of the paper is as follows.
In Section 2 we give some needed background on linear orderings and establish our notation.
Section 3 contains most of the main results in the paper.
In particular, it discusses when linear orderings have densely computable copies.
First, we prove that there is a trivialization of this concept at the $\Sigma_1$ level.

\medskip
\setcounter{section}{3}
\setcounter{theorem}{4}
\begin{theorem}
    Every linear ordering has a $\Sigma_1$-generically c.e.\ copy and a $\Sigma_1$-coarsely c.e copy.
\end{theorem}

\medskip

This lower level trivialization is parallel to theorems in \cite{CCH22} and \cite{CCHGroups} which show that every equivalence structure and Abelian $p$-group respectively have generically computable copies.
That being said, this is the first major example where it is possible to push the lower level trivialization beyond plain generic computability and into the $\Sigma_\alpha$ hierarchy.
Intuitively, this comes from the fact that linear orderings cannot code very much information in one quantifier.
We also establish a stronger result for the class of scattered linear orderings.

\setcounter{theorem}{6}
\begin{theorem}\label{scatteredThm}
    Every scattered linear ordering has a $\Sigma_2$-generically c.e.\ copy and a $\Sigma_2$-coarsely c.e copy.
\end{theorem}

We then show that any attempt to push this analysis further would be ill-fated.
In fact, we show that the question of having densely computable copies of linear orderings at later stages of the hierarchy is maximally complicated.

\setcounter{theorem}{12}
\begin{theorem}
    Fix any $\alpha\in\omega_1^{ck}$. 
    The set of linear orderings with a $\Sigma_{\alpha+2}$-generically c.e.\ copy is $\mathbf{\Sigma}_1^1$ complete.
\end{theorem}

\setcounter{theorem}{20}
\begin{theorem}
    Fix any $\alpha\in\omega_1^{ck}$.
    The set of linear orderings with a $\Sigma_{\alpha+2}$-coarsely c.e.\ copy is $\mathbf{\Sigma}_1^1$ complete.
\end{theorem}

This behavior is new in the analysis of densely computable structures.
In many cases in computable structure theory we expect the complexity of linear orderings to be intermediate between very simple structures of bounded Scott rank and fully general structures like graphs.
However, in this case of dense computability, the analysis of linear orderings is enough to bring us to the highest possible level of complexity and demonstrates new hardness results for structures in general. 
In the process of proving this reduction we produce examples of structures that have $\Sigma_{\beta}$-generically c.e.\ or coarsely c.e.\ copies but do not have $\Sigma_{\beta+1}$-generically c.e.\ or coarsely c.e.\ for $\beta$ an arbitrary computable ordinal.
Because the structures analyzed in \cite{CCH22,CCH23,CCHGroups} were all of bounded Scott rank, even this is new behavior.

At the end of the section we connect our analysis to other concepts in computable structure theory and explore some additional construction near the limit levels of the hierarchy.

In Section 4, we analyze some specific constructions of linear orderings instead of thinking about all of the copies of the structure at once.
A notable result of this section is that the concepts of generic and coarse computability do not eventually coincide at any point in the computable $\Sigma_\beta$ hierarchy.

\setcounter{section}{4}
\setcounter{theorem}{1}
\begin{proposition}
    For any $\alpha\in\omega_1^{ck}$ there is a linear ordering that is $\Sigma_{\alpha+2}$-coarsely c.e.\ but not $\Sigma_{\alpha+2}$-generically c.e.
\end{proposition}

There were analogous results seen in \cite{CCH22} for injection structures and equivalence relations at lower points in the hierarchy.
That being said, this separation at arbitrarily high computable levels is a new phenomenon. 

We also examine some intersections with ordinary computability in this section and examine dense subsets of the linear ordering given by the order type of the rationals under different enumerations.

Section 5 goes beyond linear orderings.
We show that the trivialization of linear orderings at level $\Sigma_1$ was a special feature of the theory.
In particular, it is difficult to understand when structures in general even have a $\Sigma_1$-generically c.e.\ copy.

\setcounter{section}{1}
\setcounter{theorem}{5}
\begin{theorem}[Combination of Theorems 5.2 and 5.3]
    There is a $\Pi_2$ axiomatizable class for which the set of structures with a $\Sigma_1$-generically c.e.\ copy is $\mathbf{\Sigma}_1^1$ complete. 
\end{theorem}

We also use a general form of our analysis for linear orderings to show a strong result at the lowest level of generic and coarse computability.

\setcounter{section}{5}
\setcounter{theorem}{9}
\begin{theorem}
Every model of a theory over a finite relational language has a generically computable copy and a coarsely computable copy.
\end{theorem}

\setcounter{section}{1}
This general result subsumes some of the analysis done in \cite{CCH22} about equivalence relations.

\section{Background about Linear Orderings}
Throughout this paper we will use standard notation for commonly used linear orderings.
In particular, we let $\omega$ be the order type of the natural numbers, $\zeta$ be the order type of the integers and $\eta$ be the order type of the rationals.
Given a linear ordering $L$ we write $L^*$ to be the reverse ordering.

\subsection{$\Sigma_\alpha$ Elementary Substructures of Linear Orderings}
The key part of the definition of being $\Sigma_\alpha$-generically c.e.\ or $\Sigma_\alpha$-coarsely c.e.\ is that the dense substructure witnessing these properties must be a $\Sigma_\alpha$ elementary substructure.
As a matter of notation, here and in all future cases we take $\Sigma_\alpha$ for $\alpha\in\omega_1$ to be a subset of $\L_{\omega_1,\omega}$.
This choice often makes the results of the paper more general; we will occasionally comment on intersections with first-order logic but this will be made explicit when it comes up.
It is critical that we have a general way of confirming that a linear ordering is a $\Sigma_\alpha$ substucture of another structure.
To this end, we first note a general tool that is useful for confirming that substructures are $\Sigma_\alpha$ elementary substructures.

\begin{lemma}
If $j:N\hookrightarrow M$, then $N$ is a $\Sigma_\alpha$ elementary substructure of $M$ if and only if for all $p\in N^{<\omega}$, we have $(N,p)\leq_\alpha(M,j(p))$.
\end{lemma}

\begin{proof}
This follows from Proposition 15.1 of \cite{ash2000}.
\end{proof}

The advantage of this perspective is that $(N,p)\leq_\alpha(M,j(p))$ can be confirmed using the well-known combinatorial back-and-forth game.
A key consequence of this formulation is \textit{monotonicity}.
In particular, if $\vec a\subset \vec a'$ and $(A, {\vec a'}) \leq_{\beta} (\sB, {\vec a'})$ then $(A, {\vec a}) \leq_{\beta} (\sB, {\vec a})$.
Therefore, we may always assume that our given parameters include a given finite feature. 
Furthermore, there are some previously established results for linear orderings that make this criterion particularly easy to check.
The most fundamental of these can be found in any reference on computable structure theory such as Ash and Knight's book.

\begin{lemma}[{\cite[Lemma 15.8]{ash2000}}]\label{bnf} Suppose $\A, \B$ are linear orders. Then $\A\leq_1 \B$ if and only if $\A$ is infinite or at least as large as $\B$. For $\alpha>1$, $\A\leq_\alpha \B$ if and only if for every $1\leq\beta<\alpha$ and every partition of $\B$ into intervals $\B_0,\dots,\B_n$, with endpoints in $\B$, there is a partition of $\A$ into intervals $\A_1,\dots,\A_n$ with endpoints in $\A$, such that $\B_i\leq_\beta \A_i$.
\end{lemma}

These two lemmas can be used to prove the following result.
Some definitions are needed.

\begin{definition} \begin{enumerate}
    \item[(a)] $L(x) \iff (\exists y)\ y<x$.
    \item[(b)] $R(x) \iff (\exists y)\ x<y$.
    \item[(c)] $B(x,z) \iff (\exists y)\ x < y < z$.
    \item[(d)] A substructure $\sA$ of a linear order $\sB$ \emph{preserves} a formula $P(\vec x)$ if, for any $\vec a$ from $\sA$,
    $\sA \models P(\vec a) \iff \sB \models P(\vec a)$.    
\end{enumerate}
\end{definition}

The intervals of a linear ordering are defined as follows.

\begin{definition} For any elements $a,b$ of a linear ordering $\sA$,
\begin{enumerate}
\item[(a)] $(-\infty,b) = \{x: x < b\}$
\item[(b)] $(a,\infty) = \{x: a < x\}$
\item[(c)] $(a,b) = \{x: a < x < b\}$
\end{enumerate}
\end{definition}
We sometimes adorn the interval notation above with a subscript $\sA$ if it is not clear from the context which linear ordering the interval belongs to.

\begin{proposition} \label{props1} Let $\sA \subseteq \sB$  be linear orders. Then the following are equivalent: 
\begin{enumerate}
\item[(a)] $\sA$ is a $\Sigma_1$ elementary substructure of $\sB$.
    \item[(b)] For any elements of $\sA$, the intervals in $\sA$ have the same cardinality as the intervals in $\sB$.
    \item[(c)] $\sA$ preserves the formulas $L$, $R$, and $B$ defined above. 
\end{enumerate}
\end{proposition}

Other needed results are more recent.
For example, a critical result for iterating simple constructions up the hierarchy of formulas is the following of Gonzalez and Rossegger \cite{GR23}.

\begin{lemma}
For any $\L$ and $\K$
$$\L\leq_\beta \K\implies \zeta^\alpha\cdot \L\leq_{2\alpha+\beta} \zeta^\alpha\cdot \K.$$
\end{lemma}

Note that in the above and what follows we use the conventions from Rosenstein's book \cite{Ros}.
To be explicit, multiplication of linear orderings is given by the lexicographic ordering with the second coordinate acting as the "first letter" as in Chapter 1 of \cite{Ros}.
For example, $\omega\cdot\omega^*$ is a linear ordering given by a decreasing chain of copies of $\omega$ while $\omega^*\cdot\omega$ is a linear ordering given by an increasing chain of copies of $\omega^*$.

Exponentiation of linear orderings by ordinals is more delicate issue and is treated in Chapter 8 of \cite{Ros}.
That being said, outside of ordinal bases we will only ever use it with $\zeta$ as a base.
$\zeta^\alpha$ can be thought of quite concretely.
The domain should be thought of as the functions $f:\alpha\to\zeta$ with finite support.
Functions $f$ and $g$ are compared by finding the largest value $\gamma$ on which they disagree and taking $f<g\iff f(\gamma)<g(\gamma)$.
This ordering will come up once in our discussion, and will be more explicitly discussed when it is needed.

\subsection{The Block Relations}

There are some structural aspects of linear orders that will be consistently useful to refer to.
Chief among these are the block relations that come from the analysis of Hausdorff rank. 
We begin by defining 1-blocks.

\begin{definition}
Given a linear order, $L$, let $\sim_1$ denote the \textit{1-block relation}. In other words, for $x,y\in L$ we say $x\sim_1 y$ if there are only finitely many elements between $x$ and $y$.
\end{definition}

It it a standard fact from the theory of Hausdorff rank that the 1-block relation is a convex equivalence relation, so the quotient by this relation is another linear ordering.
We call a convex suborder of $L$ closed under $\sim_1$ whose quotient by $\sim_1$ is only one point a \textit{1-block} of $L$.
It is known that the only possible isomorphism types for 1-blocks are $\omega$, $\omega^*$, $\zeta$ and $k$ for some $k\in\omega$. 

It will also be useful to define the analogous concept for higher order blocks.

\begin{definition}
    Given a linear ordering $L$ and $\beta\in\omega_1$, let $\sim_\beta$ denote the \textit{$\beta$-block relation}. This is defined inductively by saying $x\sim_{\alpha+1}y$ if $[x]_{\sim_\alpha}\sim_1 [y]_{\sim_\alpha}$ in $L/\sim_\alpha$ and for $\lambda$ a limit ordinal $x\sim_{\lambda}y$ if and only if $x\sim_{\alpha}y$ for some $\alpha<\lambda$.
\end{definition}

The quotient order $L/\sim_1$ loses quite a bit of information about $L$.
In particular, the isomorphism type of the $1$ blocks is lost.
It is sometimes important to keep track of this information separately while working in $L/\sim_1$.

\begin{definition}
We let $L[K]$ be the set of $x\in L$ with $[x]_{\sim_1}\cong K$.
\end{definition}

Note that each $L[K]$ is a linear ordering.
It inherits this as it is a substructure of $L$.

We will also occasionally make reference to the following observation of Alvir and Rossegger \cite{AR20} to define these relations within a structure.

\begin{lemma}
    The relation $\sim_\beta$ is $\Sigma_{2\beta}$ definable.
\end{lemma}

\subsection{Block Sizes and Particular Linear Ordering Constructions}

A typical way of coding a real number into a linear ordering is to realize a given subset of natural numbers as the set of finite 1-block sizes in an ordering.
In order to define this notion we first define the successor relation for linear orderings.

\begin{definition}
Given a linear ordering $L$ and $x,y\in L$ we say that $y$ is the successor of $x$ or $L\models S(x,y)$ if $L\models x < y\ \land\ \forall (z ~ z\leq x \lor z\geq y)$.
\end{definition}

Notice that $S$ is a $\Pi_1$ relation.

\begin{definition}
    Given a linear ordering $L$, let
    $$Bk(L)=\{n\in\omega\vert \exists x_1\dots\exists x_n \bigwedge_{i<n} S(x_i,x_{i+1}) \land \lnot \exists z S(z,x_{1}) \land  \lnot \exists z S(x_{n},z)\},$$
    the size of all 1-blocks realized in $L$.
\end{definition}

The fundamental computability theoretic fact about this set is that it is relatively easy to extract from a presentation of the linear ordering.

\begin{lemma}\label{lemsuc} 
For any copy of a linear ordering $L$,  $Bk(L)$ is a $\Sigma^0_3(L)$ set.  
\end{lemma}

\begin{proof} 
This is immediate from the syntactic form of the definition given above.
\end{proof}

There are various ways of creating a linear ordering with a given set of blocks.
We will need several of these types of linear ordering for our constructions and proofs.

\begin{definition}
Given a set $S$ of countable linear orderings, the \emph{shuffle sum} $Sh(S)$ of $S$ is obtained by partitioning $\eta$ into $|S|$ mutually dense sets $K_i$ and replacing each point in $K_i$ with a copy of $S_i$.
\end{definition}

If $A\subseteq\omega$ it is clear that $Bk(L)=A$.
The exact sets $A$ for which $Sh(A)$ has a computable copy were characterized by Kach \cite{Kach}.
He showed that the reals with $Sh(A)$ computable are exactly the so-called LIMINF sets.
The exact definition of this notion is not critical for the matter at hand, though we will later use the basic fact that any $\Sigma_2^0$ set is a LIMINF set.
On the other hand, the following proposition is more in line with the sort of argument central to this article, and is meant to provide a simple example of how tracking the block sizes of a linear ordering will be used. In the proposition below and many other times throughout the article we identify a natural number $n$ with the unique linear ordering of size $n$.

\begin{proposition}\label{shuffle}
There is a shuffle sum with no computable, infinite $\Sigma_2$ elementary substructure that has a c.e.\ copy, and therefore, it is not $\Sigma_2$-generically c.e.
\end{proposition}

\begin{proof}
Consider an infinite set $A\subseteq\omega\subset \mathbb{LO}$, where we identify a natural number $n$ with a finite linear order of size $n$,  and construct $Sh(A)$. Consider $j:L\hookrightarrow Sh(A)$ that is $\Sigma_2$ elementary. If $x\in j(L)$, so must be its entire (finite) 1-block as the existence of successors and predecessors is definable in two quantifiers over an element. Furthermore, the formula
$$\text{Succ}_n ~~~~:=~~~~~ \exists x_1\dots x_n \bigwedge_{i<n} \lnot \exists y~~ x_i<y<x_{i+1},$$
which states that there is a successor chain of length at least $n$, has only two quantifiers. Because $A$ is infinite, for any $n$, $Sh(A)\models \text{Succ}_n$, and therefore for any $n$, $L\models \text{Succ}_n$.

Moreover, from the above observations, $\text{Bk}(L)\subseteq_\infty A$.

Let $A$ be some infinite set with no $\Sigma_3^0$ subsets, for example the Dekker set (introduced in \cite{Dek}) of any strictly $\Pi_3^0$ set. For the sake of contradiction say that $L$ has a computable copy. This gives that $\text{Bk}(L)$ is $\Sigma_3^0$, an immediate contradiction.
\end{proof}

Note that in the above proof we do not use the infinitary nature of the $\Sigma_2$ elementary embedding $j$ at any point.
In particular, if one prefers to consider first order $\exists_2$ sentences instead of $\Sigma_2$ sentences, the same result would apply.

Similar ideas will be used when analyzing other methods of coding reals into the block sizes of linear orderings.
In particular, we will use two other main constructions.

\begin{definition}
    An \textit{$\eta$-representation} of a set $A=\{a_0,a_1,\dots\}$ is defined and denoted by:
    $$K_A:=\eta+a_0+\eta+a_1+\eta+\cdots.$$
    It is called a \textit{strong $\eta$-representation} if $a_0<a_1<\cdots.$
\end{definition}

\begin{definition}
    A \textit{$\zeta$-representation} of a set $A=\{a_0,a_1,\dots\}$ is defined and denoted by:
    $$Z_A:=\zeta+a_0+\zeta+a_1+\zeta+\cdots.$$
    It is called a \textit{strong $\eta$-representation} if $a_0<a_1<\cdots.$
\end{definition}

These notions were introduced by Lerman in \cite{LR81}.
He showed that $A$ has a computable (strong) $\zeta$-representation if and only if $A$ is $\Sigma_3^0$.
The situation with $\eta$ representations is far more complex.
Harris \cite{KHarris} showed that any $\Sigma_3^0$ set has some $\eta$-representation. 
However, the exact sets with strong $\eta$-representations have so far resisted characterization. 
That being said, we will later use the fact that all $\Pi_2^0$ sets have a strong $\eta$-representation, originally proven by Lerman \cite{LR81}.
For the purposes of this article, we consider strong representations and order the $a_i$ in increasing order for the sake of concreteness.
In particular, when we work with these orderings and we want to create non-computable behavior, we will ensure that the set $A$ is not $\Sigma_3^0$.  This choice accounts for some cases where our results are not known to be optimal, but where we think improving them may be quite technical.

\section{Linear Orderings with Densely Computable Copies}

In this section we explore which linear orderings have $\Sigma_\alpha$-generically c.e.\ copies and $\Sigma_\alpha$-coarsely c.e.\ copies.
We begin by showing some general lemmas that will help us make this determination.
We then examine the notions at the lower levels of the hierarchy where the notions trivialize.
Next, we show that at all levels of the hierarchy $\gamma=\alpha+2$ for $\alpha\in\omega_1^{ck}$, the set of linear orderings with $\Sigma_\gamma$-generically c.e.\ copies or $\Sigma_\gamma$-coarsely c.e.\ copies are $\mathbf{\Sigma}_1^1$ complete and hence as complicated as possible.
This proceeds in two steps.
We first show the $\alpha=0$ or $\gamma=2$ case, and then we provide a method to iterate this construction up the hierarchy.
The section concludes with some connections of our work with other notions from computable structure theory and a short discussion of the levels near limit ordinals omitted from our previous arguments.

\subsection{Structural Criteria for Dense Computability}
In this section we reduce the question of having a densely computable copy to a purely structural one.
The observations in this section represent a significant departure from the methods seen in previous explorations of this subject in \cite{CCH22} and \cite{CCH23}.
It is possible to recast and simplify many of the results in those papers using this new viewpoint.
Further discussions of the structural nature of our results will be saved for the final section of this paper. 
In what follows we observe that the existence of substructures with computable copies is the key criteria for understanding the existence of generically computable copies.

\begin{lemma}\label{lemce} Let $C$ be any infinite c.e.\ set. 
\begin{enumerate}

\item[(a)] Let $\sB$ be an infinite structure with a c.e.\ universe $B$, such that each function and relation of $\sB$ is computable on the set $B$.  Then $\sB$ is isomorphic to a structure with universe $C$ such that each function and relation is computable on the set $C$.

\item[(b)] Any c.e.\ linear ordering is isomorphic to a linear ordering with universe $C$ and order relation computable on the set $C$.
\end{enumerate}
\end{lemma}

\begin{proof} (a) We first show this for $C = \omega$.  

Let $B$ have one-to-one computable enumeration $\{\beta(0),\beta(1),\dots\}$ and define the structure $\sA$ with universe $\omega$ by defining, for each $n$-ary relation $R$, 
\[
R^{\sA}(i_1,\dots,i_n) \iff R^{\sB}(\beta(i_1),\dots,\beta(i_n));
\]
and, for each $n$-ary  functions $f$:
\[
Rf^{\sA}(i_1,\dots,i_n) = \beta^{-1}(f^{\sB}(\beta(i_1),\dots,\beta(i_n)).
\]

For an arbitrary infinite c.e.\ set $C$, let $\alpha$ be a computable bijection between $C$ and $\omega$ and use this bijection as above to define a structure with universe $C$ isomorphic to $\sA$ and thus to $\sB$.

\smallskip

(b) Let $A$ be a c.e.\ set and let $(A,<)$ be a c.e.\ linear ordering. Then we have $x<y \iff \neg (x=y \lor y<x)$, so that the ordering is also co-c.e., and therefore computable on $A$. The result now follows from part (a). 
\end{proof}

\begin{proposition}\label{genChar} Let  $\sA$ be an infinite structure. 
\begin{enumerate}
\item[(a)] $\sA$ has a generically computable copy if and only if it has a substructure that is isomorphic to a computable structure. 
\item[(b)] $\sA$ has a  $\Sigma_n$-generically c.e.\ copy if and only if it has a  
$\Sigma_n$ elementary substructure that is isomorphic to a c.e.\ structure.
\item[(c)] For a linear ordering $\sA$, $\sA$ has a $\Sigma_n$-generically c.e.\ copy if and only if it has a  
$\Sigma_n$ elementary substructure that is isomorphic to a computable structure.
\end{enumerate}
\end{proposition}

\begin{proof} We may assume that $\sA$ has universe $\omega$.

(a) First suppose without loss of generality that $\sA$ is generically computable and let $B$ be a (dense) c.e.\ set 
such that the substructure $\sB$ with universe $B$ has computable relations and functions.  Then $\sB$ has a computable copy by Lemma \ref{lemce}.

Next suppose that $\sB$ is a  substructure of $\sA$ which is isomorphic to a computable structure $\sC$.
Let $A$ be the universe of $\sA$ and $B$ the universe of $\sB$. 
Let $D$ be a computable dense  set such that $\omega \setminus D$ and $\omega \setminus B$ have the same cardinality.
Define the relations on $\sD$ so that $\sD$ is isomorphic to $\sB$.  Let $\Phi: \sD \to \sB$ be an isomorphism and 
extend $\Phi$ to a bijection from $\omega$ to $\omega$.  Now define the structure $\sA_1$ so that, for each relation $R(x_1,\dots,x_k)$ and any $a_1,\dots,a_k \in \omega$:

\[
\sA_1\models R(a_1,\dots,a_k) \iff \sA \models R(\Phi(a_1),\dots,\Phi(a_k)).
\]
Thus $\Phi$ will also be an isomorphism from $\sA_1$ to $\sA$ and clearly $\sD$ is a computable substructure of $\sA_1$.

\smallskip

(b)  The proof is similar. 

\smallskip

(c) This follows from (b) by Lemma \ref{lemce}.
\end{proof}

The case of coarse computability is more difficult to exactly pin down structurally.
Instead of proving an equivalence of coarse computability with structural criteria, we provide a structural tool that implies coarse computability that we will use throughout the paper.

\begin{lemma}\label{coarse.struct}
\begin{enumerate}
\item For any infinite structure $\sA$, if $\sA$ has an infinite $\Sigma_\alpha$ elementary substructure $\sB$ with a computable copy that $\Sigma_\alpha$ embeds into a computable substructure $\sC$ as a computable subset, such that 
    $$\vert \sA-\sB \vert = \vert \sC-\sB \vert$$
then $\sA$ has a $\Sigma_\alpha$-coarsely c.e.\ copy.

\item If $\sA$ is a linear ordering with a $\Sigma_\alpha$ elementary substructure $\sB$ with a computable copy and $\vert \sA-\sB \vert=n$ is finite, then $\sA$ has a $\Sigma_\alpha$-coarsely c.e.\ copy.

\end{enumerate}

\end{lemma}

\begin{proof}
\begin{enumerate}
    \item Assume that $\sA$ has a computable $\Sigma_\alpha$ elementary substructure $\sB$.
    Furthermore, let $\sC$ be a computable structure into which $\sB$ computably embeds as in the assumption of the lemma.
    There is some computable automorphism of $\sC$ that takes $\sB$ to a computable, dense subset $D$. 
    There is also some automorphism of $\sA$ that precisely sends $\sB$ to $D$ in the same order that $\sB$ is enumerated in $\sC$.
    These two presentations of $\sA$ and $\sC$ along with the witnessing set $D$ give the desired $\Sigma_\alpha$-coarsely c.e.\ copy.
    
    \item 
    Take $\sC$ to be $\sB+n$.
    Because $\sB$ is computable, so is $\sC$.
    As $\sB$ is cofinite in $\sC$ it is a computable subset of the domain.
    In particular, this structure satisfies the assumptions of part (1)  of this lemma, and therefore $\sA$ has a $\Sigma_\alpha$-coarsely c.e.\ copy.
\end{enumerate}
\end{proof}

In practice, when we apply this lemma to linear orderings we will always assume that $\vert \sA-\sB \vert$ is infinite and construct a suitable $\sC$ into which $\sB$ co-infinitely embeds. This is permissible as we have shown that the other case is easily handled in general.

\subsection{The level $\Sigma_1$}

In this section we show that every linear ordering has a $\Sigma_1$-generically c.e.\ copy and also a $\Sigma_1$-coarsely c.e.\ copy.
Based on the structural criteria described in the previous section, we begin with the following key proposition.

\begin{proposition}\label{Sigma1}
Every infinite linear order $L$ has an infinite $\Sigma_1$ elementary substructure that has a computable copy.
\end{proposition}

\begin{proof}
We break the proof into several cases based on the behavior of the 1-blocks of $L$.

Case 1: $L[\zeta]\neq\varnothing$.

Consider $a\in L[\zeta]$. We claim that $[a]_{\sim_1}$ is a $\Sigma_1$ elementary substructure of $L$. Given parameters $p\in[a]_{\sim_1}^{<\omega}$ we show that $([a]_{\sim_1},p)\leq_1 (L,p)$. This is immediate from Proposition \ref{props1}.

Case 2: $L/\sim_1$ has a first element $[a]_{\sim1}$, and $a\in L[\omega].$

We claim that $[a]_{\sim_1}$ is a $\Sigma_1$ elementary substructure of $L$. Given parameters $p\in[a]_{\sim_1}^{<\omega}$ we show that $([a]_{\sim_1},p)\leq_1 (L,p)$. Again this is immediate from Proposition \ref{props1}. 

Case 3: $L/\sim_1$ has a last element $[a]_{\sim_1}$, and $a\in L[\omega^*].$

This is symmetric to Case 2.

Case 4: There is a linear order embedding $i:\omega^* \hookrightarrow L[\omega]$.

Note that without loss of generality we can take $i(n)$ and $i(n+1)$ to be in different $\omega$ blocks as there are only finite descending sequences in any given $\omega$ block.
We claim that $\bigcup [i(n)]_{\sim_1}$ is a $\Sigma_1$ elementary substructure of $L$. Given parameters $p\in\bigcup [i(n)]_{\sim_1}$ we show that $(\bigcup [i(n)]_{\sim_1},p)\leq_1 (L/\sim_1,p)$. By Lemma \ref{bnf}, this can be confirmed by noting that the intervals between elements of $p$ are infinite or isomorphic in both structures, the first segments are infinite in both structures, and last segments are infinite in both structures.

Case 5: There is a linear order embedding $i:\omega \hookrightarrow L[\omega^*]$.

This is symmetric to Case 4.

Case 6: We are not in Cases 1-5.

Let $a\in L$ be an element in the least 1-block isomorphic to $\omega$, or let $a=\infty$ if $L[\omega]$ is empty; this exists as $L[\omega]$ is an ordinal because we are not in Case 4. Let $b\in L$ be in the greatest element 1-block isomorphic to $\omega^*$ less than $a$ or let $b=-\infty$ if $L[\omega^*]$ is empty below $a$. Such a $b$ exists as $L[\omega^*]$ is a reverse ordinal because we are not in Case 5. Because we are not in Case 2 or 3, $[a]_{\sim_1}$ is not the first element of  $L/\sim_1$ and $[b]_{\sim_1}$ is not the last element of $L/\sim_1$. Furthermore, $[a]_{\sim_1}$ cannot be the successor of $[b]_{\sim_1}$, as then $[b]_{\sim_1} + [a]_{\sim_1}\cong\zeta$ would be a single 1-block of $L$. In particular, the interval $([a]_{\sim_1},[b]_{\sim_1})$ is not empty.
Let $(a,b)_1$ denote the set of elements in $L$ that lie in equivalence classes in $([a]_{\sim_1},[b]_{\sim_1})$.
As we are not in Case 1, $(a,b)_1[\zeta]=\varnothing$, and by the definitions of $a$ and $b$, $(a,b)_1[\omega]=\varnothing$ and $(a,b)_1[\omega^*]=\varnothing$.

This leaves us with $(a,b)_1\subseteq\bigcup_{k\in\omega} L[k]$. Note that no two finite 1-blocks can be successors of one another, as they would otherwise merge into a single 1-block in $L$. This means that $(a,b)_1/\sim_1$ is dense. We consider an arbitrary single element of each 1-block in $(a,b)_1$ to obtain an embedding $\eta\hookrightarrow L$.

We claim that $\eta$ is a $\Sigma_1$ elementary substructure of $L$. Given parameters $p\in\eta$ we show that $(\eta,p)\leq_1 (L,p)$. By Lemma \ref{bnf}, this can be confirmed by noting that the intervals between elements of $p$ are infinite or isomorphic in both structures, the first segment is infinite in both structures, and last segments are infinite in both structures.

These cases are clearly exhaustive, so we have shown that every infinite linear order $L$ has an infinite $\Sigma_1$ elementary substructure that has a computable copy. In particular, we have shown that every infinite linear order $L$ has an infinite $\Sigma_1$ elementary substructure that is isomorphic to $\zeta,~\omega,~\omega^*,~\omega\cdot\omega^*, ~\omega^*\cdot\omega$ or $\eta$.
\end{proof}

We now use this result to show the desired outcome outlined in the introduction.

\begin{theorem}    Every linear ordering has a $\Sigma_1$-generically c.e.\ copy and a $\Sigma_1$-coarsely c.e.\ copy.\end{theorem}

\begin{proof}
For the generic case this follows immediately from Propositions \ref{genChar} and \ref{Sigma1}.

For the coarse case, we appeal to Lemma \ref{coarse.struct}.
In particular, we note that any among the structures 
$\zeta, ~\omega, ~\omega^*, ~\omega\cdot\omega^*, ~\omega^*\cdot\omega, ~\eta$
co-infinitely $\Sigma_1$ elementary embeds into some computable structure $N$ as a co-infinite, computable substructure.
In particular, the argument in Proposition \ref{Sigma1} directly shows that we may witness this by embedding into one of 
\[
\zeta\cdot2, ~\omega\cdot2, ~\omega^*\cdot2, ~(\omega\cdot\omega^*)\cdot2,  ~(\omega^*\cdot\omega)\cdot2, ~2\cdot\eta.
\]
   
\end{proof}

We have seen through the example produced in Proposition \ref{shuffle} that we should not be able to continue this argument any further as there are some linear orderings with no generically $\Sigma_2$ c.e.\ copies.

\subsection{$\Sigma_2$ for Scattered Linear Orderings}

A close analysis of the above argument gives a stronger result for the important subclass of \textit{scattered linear orders}.
A scattered ordering has a well defined Hausdorff rank, or equivalently takes no injection from $\eta$.
In particular, the examples of the sort produced in Proposition \ref{shuffle} are disallowed.

\begin{proposition}\label{scatteredComp}
  Every infinite scattered linear ordering $L$ has an infinite $\Sigma_2$ elementary substructure that has a computable copy.  
\end{proposition}

\begin{proof}
Let the infinite scattered linear order $L$ be given. We will construct, in several cases, an infinite $\Sigma_2$ elementary substructure $\sA$  that has a computable copy.

The argument in Case 6 of the above Proposition \ref{Sigma1} demonstrates that a scattered linear order must be in Case 1--5.
Furthermore, the embeddings defined in Cases 1--5 are already very close to being $\Sigma_2$ elementary, which is enough to demonstrate the claimed result.
We may check this one at a time. In the following, unless otherwise specified if $L/{\sim_1}$ has a first element we call it $b$ and assume that $[b]_{\sim_1}$ has a first element. In parallel, if $L/{\sim_1}$ has a last element we call it $c$ and assume that $[c]_{\sim_1}$ has a last element. If no such $b$ or $c$ exist we let $[b]_{\sim_1}$ or $[c]_{\sim_1}$ denote the empty set correspondingly.

Case 1: $L/{\sim_1}$ has a first element $a$, and $a\in L[\omega].$

We claim that $\sA = [a]_{\sim_1}+[c]_{\sim_1}$ is a $\Sigma_2$ elementary substructure of $L$.  
We will assume hereafter that  $[c]_{\sim_1}$ is nonempty and omit the simpler case that it is empty. 
Let $[c]_{\sim_1}$ have greatest element $c_0$. 
We assume by monotonicity that our parameters are $\vec a$ and include $c_0$.
We will let $c_m$ denote the least element of $[c]_{\sim_1}$ in $\vec a$.
By Lemma \ref{bnf} we only need to consider the back-and-forth relations between the intervals formed by $\vec a= \{a_0<a_1<\cdots<a_n<c_m<\cdots<c_0\}$ in both structures.
Note that the intervals of the form $(a_i,a_{i+1})$ and $(c_{i+1},c_{i})$ are finite and isomorphic in both structures.
Therefore, we only consider the back-and-forth relation $(a_n,c_m)_L\geq_2(a_n,c_m)_\sA$.
However, $(a_n,c_m)_L$ is a linear ordering with a first 1-block isomorphic to $\omega$ and $(a_n,c_m)_\sA$ is exactly the first block of $(a_n,c_m)_L$ along with the last block (should it exist and have a greatest element). 

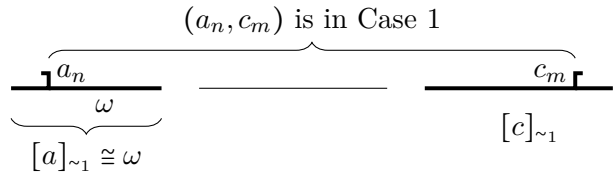
\begin{center}
\begin{tikzpicture}
  \draw[line width=1.5pt] (0,0) -- (.5,0) -- ++(0,0.2) -- ++(-0.1,0) node[midway, right] {$a_n$};
  \draw[line width=1.5pt] (.5,0) -- (2,0) node[midway, below] {$\omega$};
  \draw (2.5,0) -- (5,0);
  \draw[line width=1.5pt] (5.5,0) -- (7.5,0)node[pos=0.7, below=5pt] {$[c]_{\sim_1}$} -- ++(0,0.2) -- ++(0.1,0)node[midway, left] {$c_m$};
  \draw[line width=1.5pt] (7.5,0) -- (8,0);
  \draw[decorate, decoration={brace, amplitude=6pt, mirror}] (0,-0.4) -- (2,-0.4) node[midway, below=5pt] {$[a]_{\sim_1}\cong\omega$};
  \draw[decorate, decoration={brace, amplitude=6pt}] (.5,0.4) -- (7.5,0.4) node[midway, above=5pt] {$(a_n,c_m)$ is in Case 1};
\end{tikzpicture}
\end{center}

Therefore, we may assume that $\vec a$ is empty without loss of generality, as we have now reduced to this case. 

Given a partition $L_0,L_1,\dots,L_n$ of $L$ into intervals by $\vec p$ we construct, as follows, a partition $\sA_0,\dots,\sA_n$ of $\sA$ so that each $L_i$ is at least as large as $\sA_i$. 
Let $\vec p= \{a_0<a_1<\cdots<a_n<b_1<\cdots<b_\ell<c_m<\cdots<c_0\}$ with $a_i\in [a]_{\sim_1}$ and $c_i\in[c]_{\sim_1}$ and the $b_i$ are in niether.
We can assume that there is at least one $a_i$, $b_i$ and $c_i$ by monotonicity.
Let the elements of $[a]_{\sim_1}\cong\omega$ be identified with their natural number ordering.
We claim that the partition of $\sA$ induced by $\{0<1<\cdots<n+\ell,c_m<\cdots<c_0\}$ is our desired partition.
The intervals between the elements of $\omega$ are all empty and therefore of a less than or equal size to the corresponding $L_i$.
The intervals between the elements of $[c]_{\sim_1}$ are isomorphic in both structures and therefore of the same size.
This leaves only the intervals  $(b_\ell,c_m)_L$ and $(m+\ell,c_m)_\sA$ left to check.
However, $(b_\ell,c_m)_L$ is certainly infinite as $b_\ell\not\in [c]_{\sim_1}$ yet $c_m\in [c]_{\sim_1}$.
Therefore, it is at least as large as $(m+\ell,c_m)_\sA$ no matter what it is.

Case 2:  $L/{\sim_1}$ has a last element $a$, and $a\in L[\omega^*].$

This is symmetric to Case 2.

Case 3:  $L[\zeta]\neq\varnothing$.

Consider $a\in L[\zeta]$. We claim that $[b]_{\sim_1}+[a]_{\sim_1}+[c]_{\sim_1}$ is a $\Sigma_2$ elementary substructure of $L$. 
In this case, we can break $L$ into a left half $(-\infty,a)$ and a right half, $[a,\infty)$. The first part ends in $\omega^*$ and the second part begins with a copy of $\omega$. 
Then we can separately apply Cases 1 and 2 to these two halves as above. 

Case 4: There is a linear order embedding $i:\omega^* \hookrightarrow L[\omega]$.

Note that without loss of generality we can take $i(n)$ and $i(n+1)$ to be in different $\omega$ blocks as there are only finite descending sequences in any given $\omega$ block.
We claim that $\sA = [b]_{\sim_1}+\bigcup [i(n)]_{\sim_1}+[c]_{\sim_1}$ is a $\Sigma_2$ elementary substructure of $L$.
We will assume hereafter that  $[c]_{\sim_1}$ and $[b]_{\sim_1}$ are nonempty and omit the simpler case that they are empty. 
Let $[c]_{\sim_1}$ have greatest element $c_0$, and let $[b]_{\sim_1}$ have least element $b_0$. 
We assume by monotonicity that our parameters are $\vec a$ and include $c_0$ and $b_0$.
We will let $c_m$ denote the least element of $[c]_{\sim_1}$ in $\vec a$ and $b_k$ denote the greatest element of $[b]_{\sim_1}$ in $\vec a$.
We let $a_{n,q}$ be the $q^{th}$ element of $[i(n)]_{\sim_1}$.
By monotonicity, we assume that for every $n$ if there is some $q$ such that $a_{n,q}\in\vec a$ then  $a_{n,0}\in\vec a$.
By Lemma \ref{bnf} we only need to consider the back-and-forth relations between the intervals formed by $\vec a= \{b_0<\cdots<b_k<a_{n_0,0}<\cdots<a_{n_r,r_t}<c_m<\cdots<c_0\}$ in both structures.
Note that the intervals of the form $(b_i,b_{i+1})$, $(c_{i+1},c_{i})$ and $(a_{n_i,r_s}, a_{n_i,r_{s+1}})$  are finite and isomorphic in both structures.
This means that we only need to consider 
\begin{enumerate}
	\item the interval between $b_k$ and $a_{n_0,0}$
	\item the intervals that are between distinct $[i(n)]_{\sim_1}$,
	\item and the interval between the final element in $\bigcup [i(n)]_{\sim_1}$ denoted $a_{n_r,r_t}$ and $c_m$.
\end{enumerate}

However, the intervals of Type 2 and 3 are exactly what we have already checked in Case 1 of this very argument.
Therefore, we only need to consider the interval of Type 1.
This interval still has a linear order embedding $i:\omega^* \hookrightarrow (b_k,a_{n_0,0})[\omega]$ given by the copies of $\omega$ with index greater than or equal to $n_0$.
Therefore, we may assume that $\vec a$ is empty without loss of generality, as we have now reduced to this case. 

\smallskip

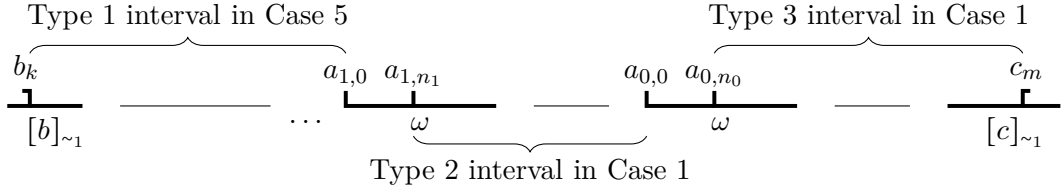
\begin{center}
    \begin{tikzpicture}
  \draw[line width=1.5pt] (0,0) -- (.3,0) -- ++(0,0.2) -- ++(-0.1,0) node[midway, above] {$b_k$};
  \draw[line width=1.5pt] (.3,0) -- (1,0) node[midway, below] {$[b]_{\sim_1}$};
  \draw (1.5,0) -- (3.5,0);
  \node at (4, -0.2) {\dots};
  \draw[line width=1.5pt] (5.4,0) -- (4.5,0) node[pos=-0.1, below=0pt] {$\omega$} -- ++(0,0.2) node[midway, above] {$a_{1,0}$};
  \draw[line width=1.5pt] (6.5,0) -- (5.4,0) -- ++(0,0.2) node[midway, above] {$a_{1,n_1}$};
  \draw (7,0) -- (8,0);
\draw[line width=1.5pt] (9.4,0) -- (8.5,0) node[pos=-0.1, below=0pt] {$\omega$} -- ++(0,0.2) node[midway, above] {$a_{0,0}$};
  \draw[line width=1.5pt] (10.5,0) -- (9.4,0) -- ++(0,0.2) node[midway, above] {$a_{0,n_0}$};
  
  \draw (11,0) -- (12,0);
  \draw[line width=1.5pt] (12.5,0) -- (13.5,0)node[pos=0.9, below=0pt] {$[c]_{\sim_1}$} -- ++(0,0.2) -- ++(0.1,0)node[midway, above] {$c_m$};
  \draw[line width=1.5pt] (13.5,0) -- (14,0);
  \draw[decorate, decoration={brace, amplitude=6pt}] (0.3,.7) -- (4.5,.7) node[midway, above=5pt] {Type 1 interval in Case 5};
  \draw[decorate, decoration={brace, amplitude=6pt, mirror}] (5.4,-.4) -- (8.5,-.4) node[midway, below=5pt] {Type 2 interval in Case 1};
  \draw[decorate, decoration={brace, amplitude=6pt}] (9.4,.7) -- (13.5,.7) node[midway, above=5pt] {Type 3 interval in Case 1};
\end{tikzpicture}
\end{center}

Given a partition $L_0,L_1,\dots,L_n$ of $L$ into intervals by $\vec p$ we construct, as follows, a partition $\sA_0,\dots,\sA_n$ of $\sA$ so that each $L_i$ is at least as large as $\sA_i$. 
Let $\vec p= \{b_0<\cdots <b_k< d_{0}< \cdots <d_{r}<c_m<\cdots<c_0\}$ with $b_i\in [b]_{\sim_1}$, $c_i\in[c]_{\sim_1}$, and $\vec d_{i}$ in between.
We can assume that there is at least one $b_i$ and $c_i$ by monotonicity.
Let the elements of $[i(0)]_{\sim_1}\cong\omega$ be identified with its natural number ordering $\{0,1,2\cdots\}$.
We claim that the partition of $\sA$ induced by 
$\{b_0<\cdots <b_k<0<1<\cdots <r<c_m< \cdots<c_0\}$
is our desired partition.
The intervals between the elements of $[b]_{\sim_1}$ are  isomorphic in both structures and therefore of the same size.
The intervals between the elements of $[c]_{\sim_1}$ are isomorphic in both structures and therefore of the same size.
The intervals between the elements $d_i$ correspond to empty intervals in $\sA$ and are therefore at least as large no matter what they are. 
This leaves only the intervals $(b_k,d_0)_L$ and $(d_r,c_m)_L$ left to check.
However, these intervals are certainly infinite in $L$ as the $d_i$ are not in the same blocks as either $b_k$ or $c_m$ by definition.
Therefore, these intervals are at least as large as the corresponding intervals in $\sA$, whatever they might be as desired.

Case 5: There is a linear order embedding $i:\omega \hookrightarrow L[\omega^*]$.

This is symmetric to Case 4.

\end{proof}

It is worth explicitly noting that the above proof is  stronger than initially advertised.
In particular, it applies to any linear ordering that is in at least one of the Cases 1--5.
Wile every scattered linear ordering is in one of these cases, many non-scattered linear orderings are as well.
At later points in this paper we will occasionally refer to this stronger version of the above proposition.

Our desired outcome in terms of dense computability follows from a similar argument as in the $\Sigma_1$ case.

\begin{theorem}
    Every scattered linear ordering has a $\Sigma_2$-generically c.e.\ copy and a $\Sigma_2$-coarsely c.e.\ copy.
\end{theorem}

\begin{proof}

For the generic case, the statement follows immediately from Propositions \ref{genChar} and \ref{Sigma1}.

For the coarse case, we appeal to Lemma \ref{coarse.struct}.
In particular, we note that any structure among 
\[ \{\{k+\zeta+k'\}_{k,k'\in\omega},
~\{\omega+k\}_{k\in\omega}, ~\omega+\omega^*, ~\{k+\omega^*\}_{k\in\omega},
 ~\{k+\omega\cdot\omega^*+k'\}_{k,k'\in\omega}, ~\{k+\omega^*\cdot\omega +k'\}_{k,k'\in\omega}
 \]
 co-infinitely $\Sigma_2$ elementarily embeds into some computable structure $N$ as a co-infinite, computable substructure.
In particular, the argument in Proposition \ref{scatteredComp} directly shows that we may witness this by embedding into 
$$\{\{k+\zeta\cdot2+k'\}_{k,k'\in\omega}, ~\{\omega\cdot2+k\}_{k\in\omega}, ~\omega\cdot2+\omega^*, ~\{k+\omega^*\cdot2\}_{k\in\omega},$$
$$\{k+(\omega\cdot\omega^*)\cdot2+k'\}_{k,k'\in\omega}, ~\{k+(\omega^*\cdot\omega)\cdot2+ k'\}_{k,k'\in\omega} \},$$
respectively.
\end{proof}

We pause here to note that we cannot push these arguments any further in the case of scattered linear orderings.

\begin{lemma}
There is an infinite, scattered linear ordering such that every linear ordering 3-equivalent to it has no c.e.\ copy.
\end{lemma}

\begin{proof}
Consider an infinite set $(a_0,a_1,\cdots)=A\subseteq\omega\subset \mathbb{LO}$, and construct the strong $\zeta$-representation
$$Z_A=\sum_\omega \zeta + a_i.$$

Note that for any $M\equiv_3 L$, $\text{Bk}(L)=\text{Bk}(M)$, as the fact that any $n\in \text{Bk}(L)$ is given by a $\Sigma_3$ criterion.
Moreover, $\text{Bk}(Z_A)=A$.

Let $A$ be any set that is not $\Sigma_3^0$. For the sake of contradiction say that $L\equiv_3SSh(A)$ has a c.e.\ copy. This gives that $\text{Bk}(L)=A$ is $\Sigma_3^0$, an immediate contradiction.
\end{proof}

Of course, this immediately implies that this linear ordering has no $\Sigma_3$-generically c.e.\ copy.
It is not too hard to show that this linear ordering also has no $\Sigma_3$-coarsely c.e.\ copy as any potential witnesses $\sD,\mathcal{E}$ would have $\mathcal{E}\equiv_3Z_A$ and therefore $\text{Bk}(\mathcal{E})=A$, meaning that $\mathcal{E}$ cannot be c.e..

\subsection{Hardness at the Level $\Sigma_2$}

In this section we analyze the set of linear orderings with $\Sigma_\beta$-generically c.e.\ and coarsely c.e.\ copies.
In particular, we demonstrate that this set of structures is $\mathbf{\Sigma}_1^1$ hard and therefore not concretely classifiable beyond the naive description of the set.

We will start our efforts by looking at the case for $\beta=2$ and then explain how to use a uniform technique to transfer our results up the hierarchy of formulas.
We will start with the case of $\Sigma_2$-generically c.e.\ and move to the case of $\Sigma_2$-coarsely c.e..
Technically speaking, the second construction will be strictly stronger than the first and will function for both purposes.
That being said, the second construction is more complex, and is  more readable once the techniques from the first construction are understood.

\begin{theorem}\label{2genHard}
The set of linear orderings with a $\Sigma_2$-generically c.e.\ copy is $\mathbf{\Sigma}_1^1$ complete.
\end{theorem}

\begin{proof}
Fix a $\Sigma^0_3$ immune set $A$.
We show that, given a linear ordering $L$, the linear ordering $(Sh(A)+\omega)\cdot L$ has a $\Sigma_2$-generically c.e.\ copy if and only if $L$ is ill-founded.
This will give the desired reduction from a complete $\mathbf{\Sigma}_1^1$ set.

First consider the case that $L$ is not well-founded.
In this case there is an embedding of $\omega\cdot\omega^*$ into $(Sh(A)+\omega)\cdot L$ that respects the 1-block relation.
Note that $(Sh(A)+\omega)\cdot L$ has no first or least element, so, as we have seen in the proof of Lemma \ref{scatteredComp}, this $\omega\cdot\omega^*$ embedding is a 2-embedding.
Proposition \ref{genChar} gives that $(Sh(A)+\omega)\cdot L$ has a $\Sigma_2$-generically c.e.\ copy.

Now we consider the case that $L$ is well founded.
Consider a structure $N$ that 2-embeds into $(Sh(A)+\omega)\cdot L$.
Let $k\in L$ be least such that there is some element $(x,k)\in N$.
Consider the set of elements $x\in (Sh(A)+\omega)$ such that $(x,k)\in N$ and call this set $S$.
Note that because $(Sh(A)+\omega)\cdot L$ has no least element, $N$ must also have no least element.
In particular, $S$ must have no least element.
If $S\subset\omega$ it will have a least element, therefore $S\cap Sh(A)\neq\emptyset$.
In fact, $S\cap Sh(A)$ cannot be finite, otherwise $S$ would still have a least element.
Note that $S$ must saturate 1-blocks as it is in the image of a 2-embedding.
As every block is finite size, this means that $S$ contains infinitely many blocks of $Sh(A)$.
Between any two blocks of $Sh(A)$ in $N$, $N$ must contain arbitrarily long successor chains by 2-elementarity as seen in the proof of Proposition \ref{shuffle}.
This means that $Bk(N)$ is an infinite subset of $A$ and is therefore not $\Sigma^0_3$.
However, if $N$ was isomorphic to a c.e.\ structure, $Bk(N)$ would be $\Sigma^0_3$.
This means that $(Sh(A)+\omega)\cdot L$ has no $\Sigma_2$ elementary substructure isomorphic to a c.e.\ structure.
In particular, it is not $\Sigma_2$-generically c.e.

The above reduction shows that the set of linear orderings with a $\Sigma_2$-generically c.e.\ copy is $\mathbf{\Sigma}_1^1$ hard.
To show it is complete we must note that it is itself $\mathbf{\Sigma}_1^1$.
This follows from writing out the ordinary definition of the concept.
$L$ has a $\Sigma_2$-generically c.e.\ copy if and only if,
$$\exists M\cong L ~~ \exists S\subset M ~~ \text{S is dense, c.e.\ and } \forall \bar{p}\in S (S,\bar{p})\equiv_2(M,\bar{p}),$$
which is $\mathbf{\Sigma}_1^1$ as being dense and c.e.\ and saying $(S,\bar{p})\equiv_2(M,\bar{p})$ are all arithmetic.
\end{proof}

There is a corresponding lightface version of the result.
In order to extract this from the above proof we just need to identify the oracle which computes the construction above.

\begin{corollary}
    The set
    $$\textbf{Gen}_2^2:=\{e~\vert~ \Phi_e^{\mathbf{0}''} \text{is a linear ordering with a } \Sigma_2\text{-generically c.e.\ copy}\}$$
    is $\Sigma_1^1(\mathbf{0}'')$ complete.
\end{corollary}

\begin{proof}
    Any strictly $\Pi_3^0$ set has a $\Sigma^0_3$ immune Dekker set.
    This means that the $A$ in Theorem \ref{2genHard} can be taken to be $\Pi^0_1$ in $\mathbf{0}''$.
    Given any $\mathbf{0}''$-computable ordering indexed by $n$ we can $\mathbf{0}''$-computably define $f(n)$ to be an ordering that multiplies the order type of $n$ by $Sh(A)+\omega$.
    This follows from the fact that $Sh(A)+\omega$ is $\mathbf{0}''$-computable by the characterization of Kach \cite{Kach}.
    The proof in Theorem \ref{2genHard} directly shows that this is a $\mathbf{0}''$-computable reduction from $\mathcal{O}$ to $\textbf{Gen}_2^2$.
\end{proof}

We note that the above construction is not sufficient to draw any conclusions about the set of linear orderings with a $\Sigma_2$-coarsely c.e.\ copy.
The construction fundamentally relies on the fact that sufficiently complex shuffle sums of finite orderings do not have $\Sigma_2$-generically computable copies as explicitly seen in Proposition \ref{shuffle}.
However, this is not the case for coarse computability.
In order to demonstrate this fact, we first show a lemma about when shuffle sums 2-embed into each other.

\begin{lemma}\label{Sh2Lem}
    If $S\subseteq T$ are infinite sets then the canonical injection $Sh(S)\hookrightarrow Sh(T)$ is a 2-elementary embedding.
\end{lemma}

\begin{proof}
    Fix some parameters in $Sh(S)$.
    By monotonicity, we may take these parameters to saturate their finite blocks.
    This means that the intervals defined between these parameters in $Sh(S)$ are all isomorphic to $Sh(S)$ and that the intervals defined by these parameters in $Sh(T)$ are all likewise isomorphic to $Sh(T)$.
    Therefore, it is enough to show that $Sh(S)\equiv_2 Sh(T)$.
    This follows from Theorem 2.7 of Gonzalez an Rossegger \cite{GR23}, as both of these structures have no largest and smallest elements and they both have arbitrarily long successor chains.
\end{proof}

We are now ready to demonstrate the shortcomings of the above construction for coarse computability.

\begin{proposition}\label{ShCoarse}
For every set $A$, $Sh(A)$ has a $\Sigma_2$-coarsely c.e.\ copy.
\end{proposition}

\begin{proof}
If $A$ is finite, $Sh(A)$ is computable so it has a computable copy and the proposition follows immediately.

We assume that $A$ is infinite.
Let $a\in A$ be the smallest value in $A$ and let $b\in A$ be the second smallest value.
Consider a computable presentation of $Sh(\omega)$ where the copies of $b$ are on a dense set.
Let $D\subset Sh(\omega)$ correspond to all blocks of size in $A-\{a\}$.
Note that this is a dense set in this particular copy.
Furthermore, $D\cong Sh(A-\{a\})$ as a substructure.
There is some presentation of $Sh(A)$ which enumerates all blocks of size in $A-\{a\}$ exactly as in $D$ and gives all leftover values to elements in blocks of size $a$.

What is left is to show that the canonical injections $D\cong Sh(A-\{a\})\hookrightarrow Sh(\omega)$ and $D\cong Sh(A-\{a\})\hookrightarrow Sh(A)$ are 2-embeddings.
However, this is exactly the content of Lemma \ref{Sh2Lem}.
This gives that the embedding of $D$ into $Sh(A)$ and $Sh(\omega)$ witnesses that this copy of $Sh(A)$ is $\Sigma_2$-coarsely c.e.
\end{proof}

This means that we require a different construction for understanding $\Sigma_2$-coarsely c.e.\ copies.
What follows is this more robust construction.

\begin{theorem}\label{2CoaHard}
 The set of linear orderings with a $\Sigma_2$-coarsely c.e.\ copy is $\mathbf{\Sigma}_1^1$ complete.
\end{theorem}

\begin{proof}
    Fix a $\Sigma^0_4$ immune set $A=\{m_1<m_2<\cdots\}$.
    Construct the strong $\eta$-representation 
    $$K_A= \sum_{i\in\omega} \eta+m_i.$$
    We show that, given a linear order $L$, the linear order $(K_A+\omega)\cdot L$ has a $\Sigma_2$-coarsely c.e.\ copy if and only if $L$ is ill-founded.
    This will give the desired reduction from a complete $\mathbf{\Sigma}_1^1$ set.

    First consider the case that $L$ is not well founded.
    In this case there is an embedding of $\omega\cdot\omega^*$ into $(K_A+\omega)\cdot L$ that respects the 1-block relation.
    Note that $(K_A+\omega)\cdot L$ has no first or least element, so, as we have in the proof of Proposition \ref{scatteredComp}, this $\omega\cdot\omega^*$ embedding is a 2-embedding.
    By the argument in Proposition \ref{scatteredComp} this gives that $(K_A+\omega)\cdot L$ has a $\Sigma_2$-coarsely c.e.\ copy.

    Now we consider the case that $L$ is well founded.
    Consider a structure $N$ that 2-embeds into $(K_A+\omega)\cdot L$.
    Let $k\in L$ be least such that there is some element $(x,k)\in N$.
    Consider the set of elements $x\in (K_A+\omega)$ such that $(x,k)\in N$ and call this set $S$.
    Note that because $(K_A+\omega)\cdot L$ has no least element, $N$ must also have no least element.
    In particular, $S$ must have no least element.
    If $S\subset\omega$ it will have a least element, therefore $S\cap K_A\neq\emptyset$.
    We aim to show far more: that $S$ contains infinitely many of the $m_i$ blocks in $K_A$.
    We consider cases to organize the proof of this subclaim.
    
    Case 1: $N$ is contained entirely in $(K_A,k)$.
    
    Note that $(K_A+\omega)\cdot L$ has arbitrarily long successor chains, so $N$ must as well.
    The only way to have arbitrarily long successor chains within a copy of $K_A$ is to contain infinitely many of the $m_i$ blocks, as desired.
    
    Case 2: $N$ has some element in $(\omega,k)$.
    
    Fix an element $y\in(K_A,k)\cap N$.
    If $N$ contains any element in $(\omega,k)$, this means that is must contain all of these elements, as $N$ saturates 1-blocks by 2-elementarity.
    In particular, it must contain the least element $m$.
    By 2-elementarity, there must be arbitrarily long successor chains between $y$ and $m$ in $N$, which means that $N$ must contain infinitely many of the $m_i$ blocks, as desired.
    
    Case 3: Otherwise.
    
    Let $p\in L$ be the least element above $k$ with some element $(x,p)\in N$.
    Note that $p$ exists because there is some element of $N$ above $(K_A,k)$ and it cannot be contained in $(\omega,k)$, so it is not in $(K_A+\omega,k)$ at all.
    Consider the set of elements $x\in (K_A+\omega)$ such that $(x,p)\in N$ and call this set $T$.
    Let $z\in (K_A+\omega,p)$ be either
    \begin{enumerate}
    	\item an element below all of the $m_i$ blocks in $T$ if one exists,
	\item else the first element of the least $m_i$ block in $T$ if one exists,
	\item else the least element of $(\omega,p)$.
    \end{enumerate}
    We must have that $z\in T$, as $y$ is not in the first two cases only if $T$ does not intersect $(K_A,p)$, in which case $T$ must saturate $(\omega,p)$ and contain its least element.
    Fix an element $y\in(K_A,k)\cap N$.
    There must be arbitrarily long successor chains between $y$ and $z$ in $N$ because there is a copy of $\omega$ between them in $L$.
    By choice of $z$, there are only (at most) length 1 successor chains under $y$ in $T$.
    This means that there must be arbitrarily long successor chains among elements with second coordinate $n<p$.
    However, by choice of $p$, this means that there must be arbitrarily long successor chains among elements with second coordinate $k$.
    Because we are not in Case 2, this means that there are infinitely many of the $m_i$ blocks in $S$, as desired.
    
    We have now shown the desired claim that $S$ contains infinitely many of the $m_i$ blocks in $K_A$ in any case.
    This can be upgraded to saying that $S$ contains cofinally many of the $m_i$ blocks.
    Consider $m_h$ and $m_{j+1}$ with $h<j$ that are successive $m_i$ blocks within $S$.
    Between these blocks there must be a successor chain of length at least $m_j$ by 2-elementarity.
    This is only possible if the unique block of length $m_j$ is actually in $S$, contradicting the fact that $m_h$ and $m_{j+1}$ are successive $m_i$ blocks in $S$.
    This means that only $m_j$ and $m_{j+1}$ can be successive $m_i$ blocks within $S$.
    Therefore, as the set $S$ contains cofinally many $m_i$ blocks, every $m_i$ block above a given $j$ must be within $S$.
    
    We now demonstrate that any $B$ that admits a 2-elementary embedding of $N$ cannot be computably enumerable.
    Let $j$ be such that every $m_i$ block greater than $j$ is within $S$.
    Consider the blocks $m_i$ and $m_{i+1}$.
    It is $\Pi_2$ to say that the ordering between the last element of $m_i$ and the first element of  $m_{i+1}$ is dense and without endpoints.
    Therefore, there must be a convex embedding of $\sum_{i>j} m_i+\eta$ into $B$.
    We can now extract the set $\{m_i\}_{i>j}$ from a presentation of $B$ by fixing the first element of the first block in the sum, $x$.
    It is $\Sigma_4^0$ in the presentation of $B$ to say that $m_i$ is the the number of successors of the $i^{th}$ element above $x$ with no predecessor. 
    If $B$ was computably enumerable, $\{m_i\}_{i>j}$ would be $\Sigma_4^0$, contradicting the choice of $A$.
    This means that every 2-elementary substructure of $(K_A+\omega)\cdot L$ cannot 2-embed into a computably enumerable structure.
    In particular, $(K_A+\omega)\cdot L$ is not $\Sigma_2$-coarsely computably enumerable, as desired.
    
    The above reduction shows that the set of linear orderings with a $\Sigma_2$-coarsely c.e.\ copy is $\mathbf{\Sigma}_1^1$ hard. 
    To show it is complete we must note that it is itself $\mathbf{\Sigma}_1^1$.
    This follows from writing out the ordinary definition of the concept.
    In the following we let $S\prec_2M$ mean that $S$ is a 2-elementary substructure of $M$, in other words $\forall \bar{p}\in S (S,\bar{p})\equiv_2(M,\bar{p})\land\forall \bar{p}\in S (S,\bar{p})\equiv_2(K,\bar{p})$.
    $L$ has a $\Sigma_2$-coarsely c.e.\ copy if and only if,
    $$\exists M\cong L ~~ \exists S\subset M ~~\exists (K,<) ~~ \text{S is dense and K is c.e.}\land S\prec_2M,$$
    which is $\mathbf{\Sigma}_1^1$ as being dense and c.e.\ and saying $(S,\bar{p})\equiv_2(M,\bar{p})$ are all arithmetic.
\end{proof}

There is also a corresponding lightface version of this result.

\begin{corollary}
    The set
    $$\textbf{Coa}_2^2:=\{e~\vert~ \Phi_e^{\mathbf{0}''} \text{is a linear ordering with a } \Sigma_2-\text{coarsely c.e.\ copy}\}$$
    is $\Sigma_1^1(\mathbf{0}'')$ complete.
\end{corollary}

\begin{proof}
    Any strictly $\Pi_4^0$ set has a $\Sigma^0_4$ immune Dekker set.
    This means that the $A$ in Theorem \ref{2CoaHard} can be taken to be $\Pi^0_2$ in $\mathbf{0}''$.
    Given any $\mathbf{0}''$-computable ordering indexed by $n$ we can $\mathbf{0}''$-computably define $f(n)$ to be an ordering that multiplies the order type of $n$ by $K_A+\omega$.
    This follows from the fact that $K_A+\omega$ is $\mathbf{0}''$-computable by a result of Lerman \cite{LR81}.
    The proof in Theorem \ref{2CoaHard} directly shows that this is a $\mathbf{0}''$-computable reduction from $\mathcal{O}$ to $\textbf{Coa}_2^2$.
\end{proof}

All of the proofs in this subsection have the special property that they do not rely at all on the infinitary choice of logic.
While first order properties are not the focus of our efforts, we still take this space to note consequences in this area.
The only properties of $\Sigma_2$-embeddings that are appealed to in the arguments are that such embeddings preserve the existence of first and last elements and preserve the $\text{Succ}_n$ properties defined in Proposition \ref{shuffle}.
These are all first order properties.
We omit the lengthy retelling of these proofs in the first order context as they are exactly the same as in the infinitary context.
That being said, we still note the analogous theorems as being true.
In the below we let $\exists_2$-generically and coarsely c.e.\ be defined in analogue with our other notions of  generically and coarsely c.e., but we only insist that the embeddings preserve first order, $\exists_2$ formulas.

\begin{theorem}\label{FirstOrder}

    \begin{enumerate}
        \item The set of linear orderings with an $\exists_2$-generically c.e.\ copy is $\mathbf{\Sigma}_1^1$ complete.
        \item  The set of linear orderings with an $\exists_2$-coarsely c.e.\ copy is $\mathbf{\Sigma}_1^1$ complete.
        \item     The set
    $$\textbf{FOGen}_2^2:=\{e~\vert~ \Phi_e^{\mathbf{0}''} \text{is a linear ordering with an } \exists_2\text{-generically c.e.\ copy}\}$$
    is $\Sigma_1^1(\mathbf{0}'')$ complete.
        \item     The set
    $$\textbf{FOCoa}_2^2:=\{e~\vert~ \Phi_e^{\mathbf{0}''} \text{is a linear ordering with a } \exists_2 \text{-coarsely c.e.\ copy}\}$$
    is $\Sigma_1^1(\mathbf{0}'')$ complete.
    \end{enumerate}
\end{theorem}

\subsection{Iterating up the Hierarchy}

We now provide the tools needed to shift the $2$ in the above proofs to higher ordinals.

\begin{proposition}\label{reduce}
    If $\zeta^\beta\cdot L$ has a $\Sigma_{2\beta+\gamma}$-elementary  substructure isomorphic to a computable structure, then $L$ has a $\Sigma_{\gamma}$-elementary substructure isomorphic to a $\mathbf{0}^{2\beta}$-computable structure.
\end{proposition}

\begin{proof}
    Let $N$ be a $\Sigma_{2\beta+\gamma}$-elementary substructure of $\zeta^\beta\cdot L$ isomorphic to a computable structure.
    If $N$ includes an element in a given $\zeta^\beta$ block it must contain all such elements by $2\beta$ elementarity.
    In particular $N\cong\zeta^\beta\cdot M$ for some $M\subset L$.
    We claim that $M$ is isomorphic to a $\mathbf{0}^{2\beta}$-computable structure and $\Sigma_{\gamma}$-elementary embeds into $L$.

    We first note that $M$ is isomorphic to a $\mathbf{0}^{2\beta}$-computable structure as $\zeta^\beta\cdot M$ is computable and $M=(\zeta^\beta\cdot M)/\sim_\beta$, where $\sim_\beta$ is $\Pi_{2\beta}$ definable. 
    We will show that this $M$ is a $\Sigma_{\gamma}$-elementary substructure of $L$.

    For any formula $\phi(\vec x)$, we can define a formula $\phi^\beta(\vec y)$ such that, for any linear ordering $L$,
for any $\vec a \in L$ and any $\vec c \in \zeta^\beta\cdot L$ such that, for each $i$, $c_i$ is an arbitrary element of the $a_i$ component of $L$,

\[
(*) \hspace{.2 in} \zeta^\beta \cdot L \models \phi^\beta(\vec c) \iff L \models \phi(\vec a)
\]

Note that this formula $(*)$ will hold for any $L$.

We define $\phi^\beta$ by recursion on formulas. 

\begin{itemize}
	\item For an atomic formula $(x_i < x_j)^\beta$ will be $(z_i < z_j)\land z_i\not\sim_\beta z_j$ and $(x_i = x_j)^\beta$ will be $z_i \sim_\beta z_j$.
	\item For a negation $(\neg \psi)^\beta$ will be $\neg \psi^\beta$.
	\item For a finitary disjunction  $(\phi \lor \psi)^\beta$ will be $(\phi^\beta \lor \psi^\beta)$.
	\item For an infinitary disjunction  $(\bigvvee_{i\in\omega}\phi_i)^\beta$ will be $\bigvvee_{i\in\omega}(\phi_i)^\beta $.
	\item If $\phi$ is $(\exists u) \psi(u, \vec x)$, then $\phi^\beta = (\exists v)\psi^\beta(v,\vec z)$.
\end{itemize}

The statement $(*)$ holds for atomic formulas by the definition of $\zeta^\beta \cdot L$. 
It is easy to see that if $(*)$ holds for $\phi$ and for $\psi$, then it will hold for $\neg \phi$ and for $\phi \lor \psi$.
Similarly, if $(*)$ holds for $\phi_i$ for each $i\in\omega$ then it will hold for  $\bigvvee_{i\in\omega}\phi_i$.
For the case of the quantifier $(*)$ will hold for $\phi$ given that it holds for $\psi$.

We observe that if $\phi$ is a $\Sigma_\gamma$ formula then $\phi^\beta$ is a $\Sigma_{2\beta+\gamma}$ formula.

Now let $\phi(\vec x)$ be an arbitrary $\Sigma_\gamma$ formula, let $\vec a \in M$, and suppose $L \models \phi(\vec a)$.
Let $\vec c \in \zeta^\beta \cdot L$ be such that $c_i$ is in the $a_i$ component of $\zeta^\beta \cdot L$. 

Then $\phi^\beta(\vec z)$ is a $\Sigma_{2\beta+\gamma}$ formula and $\zeta^\beta \cdot L \models \phi^\beta(\vec c)$. 

Since $N = \zeta^\beta \cdot M$ is a $\Sigma_{2\beta+\gamma}$ elementary substructure of $\zeta^\beta \cdot L$,
it follows that  $N  \models \phi^\beta(\vec c)$. Then by $(*)$, $M \models \phi(\vec a)$. 
Thus we have shown that $M$ is a $\Sigma_\gamma$ elementary substructure of $L$. 
\end{proof}

Note that the above proof makes essential use of the definability of $\sim_\beta$.
Standard arguments show that $\sim_\beta$ is not generally first order definable.
In particular, we are now making essential use of the infinitary logic and should not expect analogues of Theorem \ref{FirstOrder} to emerge from this section.
In particular, we leave any such exploration open in this article.

Moving forward, it will be important to recall the following result of Gonzalez and Rossegger from \cite{GR23}.
\begin{lemma}\label{lem:zpowersbfpreserve}
For any $\L$ and $\K$
$$\L\leq_\alpha \K\implies \zeta^\beta\cdot \L\leq_{2\beta+\alpha} \zeta^\beta\cdot \K.$$
\end{lemma}

To deal with parity issues we need the following results which analyze the effect of multiplying by $\eta+2+\eta$ on the back-and-forth relations as well.

\begin{lemma}\label{e2ePreserve}
Let $L=\eta+2+\eta$.
For linear orders $N,M$,
$$N\leq_\alpha M \iff L\cdot N\leq_{1+\alpha} L\cdot M.$$
\end{lemma}

\begin{proof}
We induct on $\alpha$.
The base case is immediate from Lemma \ref{bnf}.
For the successor case we first consider $N\leq_{\beta+1} M$.
We now describe a strategy for the $\exists$-player in the back-and-forth game demonstrating that $L\cdot N\leq_{1+\beta+1} L\cdot M$.
Say the $\forall$-player picks $(a_1,m_1),\dots, (a_k,m_k) \in  L\cdot M$.
The exists player will play according to an isomorphism on the first coordinate and the winning strategy $\sigma$ in the $N\leq_{\beta+1} M$ game on the second coordinate.
In other words, the $\exists$-player picks $(a_1,\sigma(m_1)), \dots, (a_k,\sigma(m_k)) \in  L\cdot N$.
It is enough to show that for all $i$, $L\cdot (\sigma(m_i),\sigma(m_{i+1}))\geq_{1+\beta} L\cdot (m_i,m_{i+1}).$
This follows immediately by induction and the fact that $\sigma$ is a winning strategy.

We now consider $L\cdot N\leq_{1+\beta+1} L\cdot M$ and describe a strategy for the $\exists$-player in the back-and-forth game demonstrating that $N\leq_{1+\beta} M$.
Say the $\forall$-player picks $m_1,\dots, m_k \in M$.
Let $a_m,b_m$ be the unique successive elements in the copy of $\eta+2+\eta$ corresponding to $m$ in $L\cdot M$.
Consider the play in the $L\cdot N\leq_{1+\beta+1} L\cdot M$ game where the $\forall$-player plays $(a_{m_1},m_1),(b_{m_1},m_1),\dots, (a_{m_k},m_k),(b_{m_k},m_k)  \in  L\cdot M$.
Letting $i$ range from $1$ to $k$, take $\sigma(a_{m_i},m_i)$ and $\sigma(b_{m_i},m_i)$ to be the play of the $\exists$-player according to the winning strategy.
Note that $\sigma(a_{m_i},m_i)$ and $\sigma(b_{m_i},m_i)$ must be successors, otherwise the game is immediately lost by the $\exists$-player.
Let $\hat{\sigma}(m_i)$ be the element of $N$ corresponding to the copy of $\eta+2+\eta$ that both $\sigma(a_{m_i},m_i)$ and $\sigma(b_{m_i},m_i)$ lie in.
The exists player in the $N\leq_{1+\beta} M$ game will now play $\hat{\sigma}(m_1),\dots, \hat{\sigma}(m_k) \in N$.
It is enough to show that for all $i$, $(\hat{\sigma}(m_i),\hat{\sigma}(m_{i+1}))\geq_{1+\beta} (m_i,m_{i+1}).$
This follows immediately by induction and the fact that $\sigma$ is a winning strategy.

For the limit case we assume that for all $\alpha<\lambda$, $N\leq_\alpha M \iff L\cdot N\leq_{1+\alpha} L\cdot M.$.
We have that 
$$N\leq_\lambda M \iff \forall \alpha<\lambda~  N\leq_\alpha M \iff \forall \alpha<\lambda~ L\cdot N\leq_{1+\alpha} L\cdot M \iff L\cdot N\leq_{1+\lambda} L\cdot M,$$
as desired.
\end{proof}

Note that for any transfinite ordinal $\alpha$, we have that $1+\alpha=\alpha$.
In particular, Lemma \ref{e2ePreserve} is at its most useful when applied to finite levels.
For our purposes, we will exclusively use the case that $\alpha=2$.
We will also need the following proposition that is analogous to Proposition \ref{reduce}.

\begin{proposition}\label{parity}
    If $(\eta+2+\eta)\cdot L$ has a $\Sigma_{3}$-elementary  substructure isomorphic to a computable structure, then $L$ has a $\Sigma_{2}$-elementary substructure isomorphic to a $\mathbf{0}''$-computable structure.
\end{proposition}

\begin{proof}
    Consider an $N$ such that $N$ is a $\Sigma_{3}$-elementary substructure of $(\eta+2+\eta)\cdot L$ isomorphic to a computable structure.
    Let $M\subseteq L$ be the set of elements $m$ such that there is a $w\in \eta+2+\eta$ with $(w,m)\in M$ and $(w,m)$ has a successor in  $(\eta+2+\eta)\cdot L$.
    By $\Sigma_3$-elementarity, if $(w,m)\in M$ then so is its successor $(w,m')$.
    Furthermore, any element $(w,m)$ produces a full copy of $\eta+2+\eta$ in $N$ as the following formulas must be satisfied:
    $$(*)~~~~~~~~~~~~N\models \exists v<(w,m)\forall p,q ~ (v<p<q<(w,m) \to \exists r ~ p<r<q),$$
    $$(**)~~~~~~~~~~~~N\models \exists v>(w,m')\forall p,q ~ (v>p>q>(w,m') \to \exists r ~ p>r>q).$$
    Lastly, given an element $s\in N$ that does not have a successor or predecessor (i.e. is not in the form $(w,m)$ or $(w,m')$), we can associate it to some successive pair.
    This is because the following formula must be satisfied:
    $$N\models \exists w<v ~ \lnot\exists p ~ w<v<p ~ \land $$
    $$(\forall q,r ~ s<q<r<w\to \exists t ~ q<t<r) \lor (\forall q,r ~ s>q>r>v\to \exists t ~ q>t>r).$$
    It follows from the above analysis that, $N=(\eta+2+\eta)\cdot M$.
    Thus $M$ has a $\mathbf{0}''$-computable presentation as the set of elements with successors in $N$ is definable in 2 quantifiers.
    We will show that this $M$ is a $\Sigma_2$-elementary substructure of $L$.
 
    For any formula $\phi(\vec x)$, we can define a formula $\phi'(\vec y)$ such that, for any linear ordering $L$,
for any $\vec a \in L$ and any $\vec c \in (\eta + 2 + \eta) \cdot L$ such that, for each $i$, $c_i$ is the unique successor in the $a_i$ component of $L$,

\[
(*) \hspace{.2 in} (\eta + 2 + \eta) \cdot L \models \phi'(\vec c) \iff L \models \phi(\vec a)
\]

Note that this formula $(*)$ will hold for any $L$.

Let $S(y)$ be the $\Sigma_2$ formula which says that $y$ is a successor. 

We define $\phi'$ by recursion on formulas. 

\begin{itemize}
	\item For an atomic formula $(x_i < x_j)'$ will be $z_i < z_j$ and $(x_i = x_j)'$ will be $z_i = z_j$.
	\item For a negation $(\neg \psi)'$ will be $\neg \psi'$.
	\item For a finitary disjunction  $(\phi \lor \psi)'$ will be $(\phi' \lor \psi')$.
	\item For an infinitary disjunction  $(\bigvvee_{i\in\omega}\phi_i)'$ will be $\bigvvee_{i\in\omega}(\phi_i)' $.
	\item Finally, if $\phi$ is $(\exists u) \psi(u, \vec x)$, then $\phi' = (\exists v) [S(v) \land \psi'(v,\vec z)]$.
\end{itemize}

The statement $(*)$ holds for atomic formulas by the definition of $(\eta + 2 + \eta) \cdot L$. 
It is easy to see that if $(*)$ holds for $\phi$ and for $\psi$, then it will hold for $\neg \phi$ and for $\phi \lor \psi$.
Similarly, if $(*)$ holds for $\phi_i$ for each $i\in\omega$ then it will hold for  $\bigvvee_{i\in\omega}\phi_i$.

We now check that in the quantifier case $(*)$ will hold for $\phi$ given that it holds for $\psi$. 

Let  $\vec a \in L$ and $\vec c \in (\eta + 2 + \eta) \cdot L$ be given such that, for each $i$, $c_i$ is the unique successor element in the $a_i$ component of $L$. 

Suppose $L \models \phi(\vec a)$, let $u \in L$ be given so that $L \models \psi(u,\vec a)$ and let $v$ be the unique successor in the $u$ component of $L$.
Then by induction $(\eta + 2 + \eta) \cdot L \models \psi'(v,\vec c)$ and therefore $(\eta + 2 + \eta) \cdot L \models \phi'(\vec c)$.

Suppose now that $(\eta + 2 + \eta) \cdot L \models \phi'(\vec z)$ and let $v$ be the successor element  such that 
$(\eta + 2 + \eta) \cdot L \models \psi'(v, \vec z)$. Let $u$ be the element of $L$ such that $v$ is the unique successor element in the $u$
 component of $(\eta + 2 + \eta) \cdot L$. Then by induction $L \models \psi(u, \vec a)$ and therefore $L \models \phi(\vec a)$.

We observe that if $\phi$ is a $\Sigma_\alpha$ formula then $\phi'$ is a $\Sigma_{1+\alpha}$ formula.

Now let $\phi(\vec x)$ be an arbitrary $\Sigma_2$ formula, let $\vec a \in M$, and suppose $L \models \phi(\vec a)$.
Let $\vec c \in (\eta + 2 + \eta) \cdot L$ such that each pair $c_i$ is the unique successor element in the $a_i$ component of $(\eta + 2 + \eta) \cdot L$. 

Then $\phi'(\vec z)$ is a $\Sigma_3$ formula and $(\eta + 2 + \eta) \cdot L \models \phi'(\vec c)$. 

Since $N = (\eta + 2 + \eta) \cdot M$ is a $\Sigma_3$ elementary substructure of $(\eta + 2 + \eta) \cdot L$,
it follows that  $N  \models \phi'(\vec c)$. Then by $(*)$, $M \models \phi(\vec a)$. 
Thus we have shown that $M$ is a $\Sigma_2$ elementary substructure of $L$.

\end{proof}

The process of formula translation established in Proposition \ref{parity} can be extended to formulas of arbitrary complexity by the exact same argument.
That being said, we only ever need to apply such a technique to $\Sigma_2$ formulas, so we omit a more explicit discussion of this matter.

We have now established the tools needed to iterate our previous argument up the ordinal hierarchy.

\begin{theorem}\label{gammaGenHard}
For any ordinal $\alpha\in\omega_1^{ck}$, the set of linear orderings with a $\Sigma_{\alpha+2}$-generically c.e.\ copy is $\mathbf{\Sigma}_1^1$ complete.
\end{theorem}

\begin{proof}
    Fix some $A$ such that every infinite subset of $A$ computes Kleene's $\mathcal{O}$.
    If $L$ is well founded, we saw in the proof of Theorem \ref{2genHard} that any $N$ that $2$-elementary embeds into $(Sh(A)+\omega)\cdot L$ computes an infinite subset of $A$ in $3$ jumps.
    In particular, there is no hyperarithmetic $N$ that $2$-elementary embeds into $(Sh(A)+\omega)\cdot L$.
    By Proposition \ref{reduce} it follows that $$\zeta^\beta\cdot(Sh(A)+\omega)\cdot L$$ does not have any hyperarithmetic $\Sigma_{2\beta+2}$ elementary substructures.
    Similarly, by Proposition \ref{parity} and Proposition \ref{reduce} it follows that $$\zeta^\beta\cdot(\eta+2+\eta)\cdot(Sh(A)+\omega)\cdot L$$ does not have any hyperarithmetic $\Sigma_{2\beta+3}$ elementary substructures.
    For ease of notation, given a computable ordinal $\gamma=\alpha+2$, let $$L_\gamma=
    \begin{cases}
    \zeta^\beta\cdot(Sh(A)+\omega)\cdot L,& \text{if } \gamma=2\beta+2\\
    \zeta^\beta\cdot(\eta+2+\eta)\cdot(Sh(A)+\omega)\cdot L,              & \text{if } \gamma=2\beta+3.
\end{cases}$$ 
    The above analysis implies that if $L$ is well founded, then $L_\gamma$ has no hyperarithemic $\Sigma_\gamma$ elementary substructure, and therefore has no $\Sigma_\gamma$-generically c.e.\ copy.

    To establish the desired reduction, all that is left to show is that any ill-founded $L$ yields an $L_\gamma$ with a $\Sigma_\gamma$-generically c.e.\ copy.
    In this case, there is a $2$-embedding of $\omega\cdot\omega^*$ into $(Sh(A)+\omega)\cdot L$.
    It follows at once from Lemmas \ref{lem:zpowersbfpreserve} and \ref{e2ePreserve} that $L_\gamma$ accepts a $\gamma$-embedding from $\zeta^\beta\cdot(\omega\cdot\omega^*)$ or $\zeta^\beta\cdot(\eta+2+\eta)\cdot(\omega\cdot\omega^*)$ depending on the parity of $\gamma$.
    Because $\gamma$ is a computable ordinal, $\beta$ is too.
    This means that these linear orderings are also computable.
    In turn, $L_\gamma$ has a $\Sigma_\gamma$-generically c.e.\ copy by Proposition \ref{genChar}.

    In total we have that for any $L$ and any computable $\gamma=\alpha+2$, $L$ is ill-founded if and only $L_\gamma$ has a $\Sigma_\gamma$-generically computable copy.
    This gives that the set of linear orderings with a $\Sigma_\gamma$-generically computable copy is $\mathbf{\Sigma}_1^1$ hard.
    To show it is complete we must note that it is itself $\mathbf{\Sigma}_1^1$.
    This follows from writing out the ordinary definition of the concept.
    $L$ has a $\Sigma_\gamma$-generically c.e.\ copy if and only if,
    $$\exists M\cong L ~~ \exists S\subset M ~~ \text{S is dense, c.e.\ and} \forall \bar{p}\in S (S,\bar{p})\equiv_\gamma(M,\bar{p}),$$
    which is $\mathbf{\Sigma}_1^1$ as being dense and c.e.\ and saying $(S,\bar{p})\equiv_\gamma(M,\bar{p})$ are all hyperarithmetic.
\end{proof}

A similar argument shows the analogous result for coarse computability. 

\begin{theorem}\label{gammaCoaHard}
For any ordinal $\alpha\in\omega_1^{ck}$, the set of linear orderings with a $\Sigma_{\alpha+2}$-coarsely c.e.\ copy is $\mathbf{\Sigma}_1^1$ complete.
\end{theorem}

\begin{proof}
    Fix some $A$ such that every infinite subset of $A$ computes Kleene's $\mathcal{O}$.
    If $L$ is well founded, we saw in the proof of Theorem \ref{2CoaHard} that any $N$ that $2$-embeds into $(K_A+\omega)\cdot L$ can only $2$-embed into structures $M$ which compute an infinite subset of $A$ in $4$ jumps.
    In particular, there is no hyperarithmetic $M$ that shares a $2$-elementary substructure with $(K_A+\omega)\cdot L$.
    By Proposition \ref{reduce} it follows that $$\zeta^\beta\cdot(K_A+\omega)\cdot L$$ does not have any $\Sigma_{2\beta+2}$-elementary substructures in common with any hyperarithmetic $M$.
    Similarly, by Proposition \ref{parity} and Proposition \ref{reduce} it follows that $$\zeta^\beta\cdot(\eta+2+\eta)\cdot(K_A+\omega)\cdot L$$ does not have any $\Sigma_{2\beta+3}$-elementary substructures in common with any hyperarithmetic $M$.
    For ease of notation, given a computable ordinal $\gamma=\alpha+2$, let $$L_\gamma=
    \begin{cases}
    \zeta^\beta\cdot(K_A+\omega)\cdot L,& \text{if } \gamma=2\beta+2\\
    \zeta^\beta\cdot(\eta+2+\eta)\cdot(K_A+\omega)\cdot L,              & \text{if } \gamma=2\beta+3.
\end{cases}$$ 
    The above analysis implies that if $L$ is well founded, then $L_\gamma$ shares no $\Sigma_\gamma$ elementary substructure with a structure with a hyperarithmetic copy, and therefore has no $\Sigma_\gamma$-generically c.e.\ copy.

    To establish the desired reduction, all that is left to show is that any ill-founded $L$ yields an $L_\gamma$ with a $\Sigma_\gamma$-coarsely c.e.\ copy.
    In this case, there is a $2$-embedding of $\omega\cdot\omega^*$ into $(K_A+\omega)\cdot L$.
    It follows at once from Lemmas \ref{lem:zpowersbfpreserve} and \ref{e2ePreserve} that $L_\gamma$ accepts a $\gamma$-embedding from $\zeta^\beta\cdot(\omega\cdot\omega^*)$ or $\zeta^\beta\cdot(\eta+2+\eta)\cdot(\omega\cdot\omega^*)$ depending on the parity of $\gamma$.
    Because $\gamma$ is a computable ordinal, $\beta$ is too.
    This means that these linear orderings are also computable.
    In turn, $L_\gamma$ has a $\Sigma_\gamma$-coarsely c.e.\ copy.
    This follows from Lemma \ref{coarse.struct} and the fact that $\zeta^\beta\cdot(\omega\cdot\omega^*)$ and $\zeta^\beta\cdot(\eta+2+\eta)\cdot(\omega\cdot\omega^*)$ coninfinitely, computably $\gamma$-embed into $(\zeta^\beta\cdot(\omega\cdot\omega^*))\cdot2$ and $(\zeta^\beta\cdot(\eta+2+\eta)\cdot(\omega\cdot\omega^*))\cdot2$ respectively.

    In total we have that for any $L$ and any computable $\gamma=\alpha+2$, $L$ is ill-founded if and only $L_\gamma$ has a $\Sigma_\gamma$-coarsely computable copy.
    This gives that the set of linear orderings with a $\Sigma_\gamma$-coarsely computable copy is $\mathbf{\Sigma}_1^1$ hard.
    To show it is complete we must note that it is itself $\mathbf{\Sigma}_1^1$.
    This follows from writing out the ordinary definition of the concept.
    $L$ has a $\Sigma_\gamma$-generically c.e.\ copy if and only if,
    $$\exists M\cong L ~~ \exists S\subset M ~~\exists (K,<) ~~ \text{S is dense and K is c.e.}\land \forall \bar{p}\in S (S,\bar{p})\equiv_\gamma(M,\bar{p})\land\forall \bar{p}\in S (S,\bar{p})\equiv_\gamma(K,\bar{p}) ,$$
    which is $\mathbf{\Sigma}_1^1$ as being dense and c.e.\ and saying $(S,\bar{p})\equiv_\gamma(M,\bar{p})$ are all hyperarithmetic.
\end{proof}

The constructions above were written with the use of a $\Pi_1^1$ complete set for the convenience of uniform notation.
That being said, this was not strictly needed.
There are lightface analogues to the results above that avoid such complicated sets but use the exact same argument.

\begin{corollary}
    For $\gamma=\alpha+2$, the sets
    $$\textbf{Gen}_\gamma^{\gamma+1}:=\{e~\vert~ \Phi_e^{\mathbf{0}^{(\gamma+1)}} \text{is a linear ordering with a } \Sigma_\gamma-\text{ generically c.e.\ copy}\}$$
    and
    $$\textbf{Coa}_\gamma^{\gamma+1}:=\{e~\vert~ \Phi_e^{\mathbf{0}^{(\gamma+1)}} \text{is a linear ordering with a } \Sigma_\gamma-\text{ coarsely c.e.\ copy}\}$$
    are $\Sigma_1^1(\mathbf{0}^{(\gamma+1)})$ complete.
\end{corollary}

\begin{proof}
    We begin with the generic case.
    Let $A$ be a $\Pi_{\gamma+2}$ set with no infinite $\Sigma_{\gamma+2}$ subset.
    It follows by relativizing the result of Kach \cite{Kach} that the linear ordering $Sh(A)+\omega$ is computable in $\mathbf{0}^{(\gamma+1)}$.
    Therefore, the multiplication map $L\mapsto L_\gamma$ defined in Theorem \ref{gammaGenHard} is computable in $\mathbf{0}^{(\gamma+1)}$.
    It follows from Propositions \ref{reduce} and \ref{parity} that as $(Sh(A)+\omega)\cdot L$ has no $\mathbf{0}^{(\gamma+1)}$-computable 2-elementary substructures, $L_\gamma$ has no computable $\gamma$-elementary substructures.
    Therefore, $L_\gamma$ is not $\Sigma_\gamma$-generically computable in this case.
    From here, the exact argument from Theorem \ref{gammaGenHard} finishes the proof.

    We now consider the coarse case.
    Let $A$ be a $\Pi_{\gamma+3}^0$ set with no infinite $\Sigma_{\gamma+3}^0$ subset.
    It follows by relativizing the result of Lerman \cite{LR81} that the linear ordering $K_A+\omega$ is computable in $\mathbf{0}^{(\gamma+1)}$.
    Therefore, the multiplication map $L\mapsto L_\gamma$ defined in Theorem \ref{gammaCoaHard} is computable in $\mathbf{0}^{(\gamma+1)}$.
    It follows from Propositions \ref{reduce} and \ref{parity} that as $(K_A+\omega)\cdot L$ shares no $\mathbf{0}^{(\gamma+1)}$-computable 2-elementary substructures with a $\mathbf{0}^{(\gamma+1)}$-computable structure, $L_\gamma$ shares no computable $\gamma$-elementary substructures with a computable structure.
    From here, the exact argument from Theorem \ref{gammaCoaHard} finishes the proof.
\end{proof}

We note that there may be a way to further optimize these results with respect to the oracle being used in the reduction.
This is not the focus of this paper, but could be an interesting place for future research.

It is worth pointing out that in the process of proving the above results we have produced examples of structures which have novel dense computability properties even when analyzed in isolation.

\begin{proposition}
    Fix an $\alpha\in\omega_1^{ck}.$
    There is a linear ordering with a $\Sigma_{\alpha+1}$-generically c.e.\ copy but no $\Sigma_{\alpha+2}$-generically c.e.\ copy.
    There is a linear ordering with a $\Sigma_{\alpha+1}$-coarsely c.e.\ copy but no $\Sigma_{\alpha+2}$-coarsely c.e.\ copy.
\end{proposition}
    
\begin{proof}
    Fix a well order $L$. Consider $L_{\alpha+2}$ as defined in Theorem \ref{gammaGenHard}.
    We saw in the proof of Theorem \ref{gammaGenHard} that $L_{\alpha+2}$ does not have a $\Sigma_{\alpha+2}$-generically c.e.\ copy.
    To find a $\Sigma_{\alpha+1}$-generically c.e.\ copy we find a $\Sigma_{\alpha+1}$ substructure with a computable copy and apply Proposition \ref{genChar}.
    Note that $\eta$ is a $\Sigma_1$ elementary substructure of $(Sh(A)+\omega)\cdot L$ as in the argument from Theorem \ref{Sigma1}.
    It follows from Lemmas \ref{lem:zpowersbfpreserve} and \ref{parity} that, depending on parity, either $\zeta^\beta\cdot\eta$ or $\zeta^\beta\cdot(\eta+2+\eta)\cdot\eta$ is a a $\Sigma_{\alpha+1}$ elementary substructure of $L_{\alpha+2}$.
    As this structure is computable, we have demonstrated the desired claim.

    To see the analogous result for coarse computability we instead fix $L_{\alpha+2}$ as defined in Theorem \ref{gammaCoaHard}.
    We saw in the proof of Theorem \ref{gammaCoaHard} that $L_{\alpha+2}$ does not have a $\Sigma_{\alpha+2}$-coarsely c.e.\ copy.
    To find a $\Sigma_{\alpha+1}$-coarsely c.e.\ copy we find a $\Sigma_{\alpha+1}$ substructure with a computable co-infinite embedding into a computable structure and apply Lemma \ref{coarse.struct}.
    Note that $\eta$ is a $\Sigma_1$ elementary substructure of $(K_A+\omega)\cdot L$ as in the argument from Theorem \ref{Sigma1}.
    It follows from Lemmas \ref{lem:zpowersbfpreserve} and \ref{parity} that, depending on parity, either $\zeta^\beta\cdot\eta$ or $\zeta^\beta\cdot(\eta+2+\eta)\cdot\eta$ is a a $\Sigma_{\alpha+1}$ elementary substructure of $L_{\alpha+2}$.
    As this structure is computable and computably embeds co-infinitely into itself, we have demonstrated the desired claim.
\end{proof}

\subsection{Connections with Other Notions in Computable Structure Theory}
Another way to understand these results is to consider them as obstructions to a potential theory of $\alpha$-embedability.
This can be understood in contrast with the back-and-forth relations.
Note that if $N$ is $\alpha$-embedable into $M$ then it follows that $N\geq_{\alpha+1} M$.
One might then imagine that the condition of being $\alpha$-embedable acts quite similarly to the $\geq_{\alpha+1}$.
In particular, the following result holds for the back-and-forth relations.

\begin{lemma}
    (Lemma VI.14 in \cite{cst2}) 
    For every $\alpha\in\omega_1$ and structure $\A$, there are $\Pinf{2\alpha}$ formulas $\phi_{\alpha,\A}$ and $\psi_{\alpha,\A}$ such that 
    $$\B\models\phi_{\alpha,\A} \iff\A\leq_\alpha\B,$$
    and 
    $$\B\models\psi_{\alpha,\A} \iff\A\geq_\alpha\B.$$
\end{lemma}

An analogous theory of $\alpha$-embedability would enjoy a similar result.
However, it follows from our above theorems that this is not the case.

\begin{proposition}
    For $\alpha\in\omega_1^{ck}$ that is not a limit or sucessor of a limit, there are structures $\A$ with no infinitary $\phi_{\alpha,\A}$ such that $\B\models\phi$ if and only if $\A$ $\alpha$-embeds into $\B$.
\end{proposition}

\begin{proof}
    For the sake of contradiction, say that such $\phi_{\alpha,\A}$ always exists.
    Then there is an infinitary formula that classifies exactly when you have a computable, infinite $\alpha$-elementary substructure.
    This formula is obtained by taking a countable disjunction over $\phi_{\alpha,\A}$ for every computable, infinite $A$.
    By Lemma \ref{genChar} this means there is an infinitary formula that classifies exactly when a structure has a $\Sigma_\alpha$-generically computable copy.
    This is a contradiction to Theorem \ref{gammaGenHard}.
\end{proof}

The above result is unsurprising upon reflection.
The theory of simple embedability is similarly plagued as having $\omega^*$ embedable is the same as being ill-founded.
On the other hand, the above proof outlines a construction that shows that being $\geq_{\alpha+1}$ an infinite, computable structure is always $\Sigma_{2\alpha+3}$.
This means that the condition of being $\alpha$-embedable acts quite differently to $\geq_{\alpha+1}$ for any sufficiently rich class of structures.

We can also understand our result in the context of the (boldface) Pair of Structures theorem:
\begin{theorem}
(Lemma VII.30 in \cite{cst2}) Say that $\A$ and $\B$ are structures in the same signature. $\A\leq_\alpha\B$ if and only if distinguishing $\A$ from $\B$ is $\mathbf{\Sigma}^0_\alpha$ hard.
\end{theorem}

In particular, if we have two isomorphism invariant sets of structures $S$ and $T$ that we want to show are $\mathbf{\Sigma}^0_\alpha$ hard to separate it is enough (and quite typical) to find $\A\in S$ and $\B\in T$ such that $\A\leq_\alpha\B$.
We know from Theorems \ref{2CoaHard} and \ref{2genHard} that the set of linear orderings with a $\Sigma_2$-coarsely c.e.\ copy and the set of linear orderings with a $\Sigma_2$-generically c.e.\ copy are both $\mathbf{\Sigma}^0_\alpha$-hard for every $\alpha$.
It is worth pointing out, however, that we can witness this weaker property by looking at individual structures as in the the Pair of Structures theorem.

\begin{proposition}
    For every $\alpha$, there are two linear orderings $L$ and $K$ such that $L\leq_\alpha K$ and $L$ has a $\Sigma_2$-generically c.e.\ copy while $K$ does not.
    There are also two linear orderings $L'$ and $K'$ such that $L'\leq_\alpha K'$ and $L'$ has a $\Sigma_2$-coarsely c.e.\ copy while $K'$ does not.
\end{proposition}

\begin{proof}
    Take a well founded $W$ and and ill-founded $I$ with $I\equiv_\alpha W$ such as $I=\omega^\alpha\cdot(1+\omega^*)$ and $W=\omega^\alpha$ (see Lemma II.38 in \cite{cst2}). 
    Note that $$(Sh(A)+\omega)\cdot I\equiv_\alpha (Sh(A)+\omega)\cdot W$$ as the $\exists$-player can simply play isomophically on the $Sh(A)+\omega$ coordinate and according to their winning strategy in $I\equiv_\alpha W$ on the second coordinate.
    By the same reasoning, $$(K_A+\omega)\cdot I\equiv_\alpha (K_A+\omega)\cdot W.$$
    The arguments in Theorems \ref{2CoaHard} and \ref{2genHard} show that $L=(Sh(A)+\omega)\cdot I$, $K=(Sh(A)+\omega)\cdot W$, $L'=(K_A+\omega)\cdot I$ and $K'=(K_A+\omega)\cdot W$ have the desired properties.
\end{proof}

We can look at the the case that $S$ contains all structures with a computable copy to get a result that is in line with the above result for generic computability up to a computable level.

\begin{proposition}
The set of linear orders with a computable copy is $\mathbf{\Sigma}^0_\alpha$ hard for every $\alpha\in\omega_1^{ck}$.
\end{proposition}

\begin{proof}
Note that, by the proof of Theorem 2.7 in \cite{GR23}, $Sh(A)\leq_3 \zeta$ for any infinite set $A$.
Fix $\alpha\in\omega_1^{ck}$ and take $A=\mathbf{0}^{(\alpha+\omega)}$.
Lemma \ref{lem:zpowersbfpreserve} gives that $\zeta^\alpha \cdot Sh(A)\leq_{2\alpha+3} \zeta^{\alpha+1}$.
We know that $\zeta^\alpha \cdot Sh(A)$ cannot have a computable copy by Proposition \ref{reduce}.
Of course, $\zeta^{\alpha+1}$ does have a computable copy.
This gives that the set of structures with computable copies is at least $\mathbf{\Sigma}^0_{2\alpha+3}$ hard, as desired.
\end{proof}

It should be noted that, despite the above result the set of a linear orders with a computable copy is actually Borel.

\begin{proposition}
The set of linear orders with a computable copy is Borel.
\end{proposition}

\begin{proof}
There are countably many such computable linear orders and each must have a $\Pinf{\omega_1^{ck}+2}$ Scott sentence by Nadel's theorem (Theorem VI.15 in \cite{cst2}).
Therefore, the set can be seen to be $\Sinf{\omega_1^{ck}+3}$
\end{proof}

There is still some daylight between this and the above result.
One immediate example higher up would be $\omega_1^{ck}$ and the Harrison linear ordering, which are $\omega_1^{ck}$ equivalent yet only one of them is computable.
A technique other than the pair of structure theorem would need to be used to push this analysis much further, as the  $\Pinf{\omega_1^{ck}+2}$ equivalence class determines the isomorphism type of any structure with a computable copy.

Notably this argument does not transfer to having a $\Sinf{\alpha}$-generically c.e.\ copy, as we would expect from the above result.
In particular, for any fixed computable $\alpha$ there are linear orders with arbitrarily large Scott rank with a $\Sinf{\alpha}$-generically c.e.\ copy.
We can see this by noting that arbitrarily large ordinals take an $\alpha$-embedding from the computable $\omega^\alpha$ for any $\alpha\in\omega^{ck}_1$.
This can be seen by a straightforward adaption of the classic arguments of Ash (see e.g.\cite{cst2} Lemma II.38).

\subsection{Additional Constructions and Questions Near Limit Levels}
The previous subsections had very little to say about limit levels and the successor of limit levels.
These ordinals are not our focus, but it is still possible to find interesting behavior at these levels as we demonstrate here.
We also provide some open questions to be explicit about exact extent of the analysis of the previous sections.

In the following proposition and only in the following proposition we use the notation $Z_\gamma$ for $\gamma\in\omega_1$ to denote the unique proper initial segment of $\zeta^\gamma$ up to isomorphism, following the conventions in \cite{GHKMMS}.
It is well known and explicitly pointed out in \cite{GHKMMS} that $Z_{\beta+1}=\zeta^\beta\cdot\omega^*+Z_{\beta}.$
We also use the notation $S_{\delta}^k(x)=y$ to mean that $[y]_{\sim_\delta}$ is the $k^{th}$ successor of $[x]_{\sim_\delta}$ in the quotient order.
The observation of Alvir and Rossegger \cite{AR20} shows that $S_{\delta}^k(x)=y$ is a $\Sigma_{2\delta+2}$ relation and is therefore definable below the smallest limit level larger than $\delta$.

\begin{proposition}
   For limit ordinals $\lambda<\omega_1^{ck}$, there is a linear order that has no $\Sigma_{\lambda}$-generically c.e.\ copy, but has $\Sigma_{\alpha}$-generically c.e.\ copies for all $\alpha<\lambda$ and similar for coarse computability. 
\end{proposition}

\begin{proof}
Fix a computable fundamental sequence of successor ordinals $\delta_i\to\lambda$. Given a real number $A=(a(0)<a(1)<\cdots)$ form the following order
$$Z(A)=\sum_{i\in\omega}\zeta^{\delta_{a(i)}}.$$

We first show that $Z(A)$ has $\Sigma_{\alpha}$-generically computable copies for $\alpha<\lambda$.
Fix $n$ such that $\delta_{a(n)}>\alpha$ and let $\delta_{a(n)}=\beta+1$.
Note that 
$$Z(A,n)=\sum_{i\leq n}\zeta^{\delta_{a(n)}},$$
is a finite sum of computable structures and therefore is computable.
Consider the initial embedding of $Z(A,n)$ into $Z(A)$.
We show this is $\Sigma_{\alpha}$ which is enough to demonstrate the claim by Proposition \ref{genChar}.
Consider parameters $p=(p_1,\cdots,p_n)\in Z(A,n)\subset Z(A)$.
Without loss of generality, $p_n$ is in the last summand of $Z(A,n)$.
We show that $(Z(A,n),p)\equiv_{\alpha} (Z(A),p)$.
This is the same as showing that $(p_i,p_{i+1})_{Z(A,n)}\equiv_{\alpha}(p_i,p_{i+1})_{Z(A)}$ for all $0\leq i\leq n+1$, where $p_0=-\infty$ and $p_{n+1}=\infty$.
Because of the selection of embedding, the intervals are isomorphic for $0\leq i<n+1$.
Therefore, we need only show that
$$Z_{\delta_{a(n)}}^* = Z_\beta^*+\zeta^\beta\cdot\omega\equiv_\alpha Z_\beta^*+\zeta^\beta\cdot\omega + \sum_{i>n}\zeta^{\delta_{a(i)}}.$$
By factoring out a copy of $\zeta^\beta$ from the right hand side, we observe that it is enough to show that, for some order $K$, 
$$\zeta^\beta\cdot\omega\equiv_\alpha\zeta^\beta\cdot(\omega+K).$$
As $\beta\geq\alpha$, this follows immediately from Lemma \ref{lem:zpowersbfpreserve}.

To see the analogous statement for Coarse computability, by Lemma \ref{coarse.struct} it is enough to observe that the same argument as above shows that $Z(A,n)$ $\alpha$-embeds into $Z(A,n+1)$ computably.

We now show that $Z(A)$ has no $\Sigma_{\lambda}$-generically c.e.\ copy or  $\Sigma_{\lambda}$-coarsely c.e.\ copy so long as $A$ is picked correctly.
Consider $N$ that is $\Sigma_{\lambda}$ elementary equivalent to $Z(A)$.
Define
$$\psi_n=~ \exists x ~ \bigwwedge_{k\in\omega} \exists y ~ S_{\delta_{a(n)}-1}^k(x)=y \land \lnot\exists z ~ S_{\delta_{a(n)}}^k(x)=y,$$
and
$$Zpow(L)=\{n ~\vert ~L\models \psi_n\}.$$
By construction, $Zpow(Z(A))=A$.
Furthermore, as each $\psi_n$ has quantifier rank below $\lambda$, $Zpow(N)=A$.
Note that for any copy of $L$, $Zpow(L)\leq_T L^{\lambda}$.
Therefore if we let $A=\mathbf{0}^{(\lambda+\omega)}$, $N$ cannot be computable.
However, there must be an $N$ $\lambda$-equivalent to $Z(A)$ with a computable presentation if $Z(A)$ actually has a  $\Sigma_{\lambda}$-generically c.e.\ copy or  $\Sigma_{\lambda}$-coarsely c.e.\ copy.
This means that there is no such copy, as we desired.
\end{proof}

This only leaves the following case which we leave open.

\begin{question}
For limit ordinals $\lambda$, is there a linear order that has a $\Sigma_{\lambda}$-generically computable copy, but no $\Sigma_{\lambda+1}$-generically computable copy?
\end{question}

The previous result can even be pushed to the level of $\omega_1^{ck}$, but the construction is quite different.

\begin{proposition}
There is a linear order that has a $\Sigma_\alpha$-generically c.e.\ copy for every computable $\alpha$, but no $\Sigma_\beta$-generically c.e.\ copy for non-computable $\beta$. We have a similar statement for coarsely c.e.\ copies.
\end{proposition}

\begin{proof}
Let $L=\omega_1^{ck}$.
We first note that for all computable $\alpha$, $L$ has a $\Sigma_{\alpha}$-generically computable copy. 
Let $\omega^\alpha$ embed initially in $L$.
Pick parameters $p\in \omega^\alpha\subset L$.
All of the intervals determined by the parameters are isomorphic in both structures, except for the final interval.
However, for all $x\in \omega^\alpha$, $\omega^\alpha_{\geq x}\cong \omega^\alpha$ and $L_{\geq x}\cong L$.
Therefore, we need only show that $L\equiv_{\alpha}\omega^\alpha$.
However, this follows from a well known theorem of Ash (Lemma II.28 \cite{cst2}).
To demonstrate the analogue for coarse computability, note that $\omega^\alpha$ computably $\alpha$-embeds initially into $(\omega^\alpha)\cdot2$.

On the other hand, consider $N$, an $\omega_1^{ck}$ elementary substructure of $L$.
Note that all substructures of $L$ are ordinals, so $N$ is an ordinal.
If $N$ is a computable ordinal it has Scott rank less than $\omega_1^{ck}$. However, this yields that $N\cong\omega_1^{ck}$ by $\omega_1^{ck}$ elementarity, an immediate contradiction.
Thus, $N\cong\omega_1^{ck}$ and it is cofinally embedded into $L$.
By Ash's theorem, for any $n\in N$ it can be said in computably many quantifiers that $N_{<n}\cong L_{<n}.$
This forces that $N=L$.
This means that $L$ cannot have a $\omega_1^{ck}$-generically c.e.\ copy or a $\omega_1^{ck}$-coarsely c.e.\ copy.
\end{proof}

There is a sense in which this example is pushing how high we may witness differences in notions of generically computable.

\begin{proposition}
For any structure $M$, if $M$ has a $\Sigma_{\omega_1^{ck}+2}$ elementary substructure that is computable (c.e.), it has a computable (respectively c.e.) copy.
\end{proposition}

\begin{proof}
Consider a computable or c.e.\ witness $N\hookrightarrow M$ to $\Sigma_{\omega_1^{ck}+2}$ elementarity. By Nadel's theorem (Theorem VI.19 in \cite{cst2}) $\SR(N)\leq \omega_1^{ck}+1$. It follows that $N\cong M$. 
\end{proof}

This means that by Proposition \ref{genChar} there are no linear orderings with a $\Sigma_{\omega_1^{ck}+2}$-generically c.e.\ copy that do not also have a computable copy.
For structures in general, having a generically c.e.\ copy aligns exactly with having a c.e.\ copy.
There are a number of interesting questions this high up in the heirarchy that we do not explore here.
Answering these questions in a positive fashion would necessarily involve constructing orderings of high Scott rank, and represent a promising direction for future research. 

\begin{question}
Is there a linear order that has a $\Sigma_{\omega_1^{ck}}$-generically c.e.\ copy, but no $\Sigma_{\omega_1^{ck}+1}$-generically c.e.\ copy?
\end{question}

\begin{question}
Is there a linear order that has a $\Sigma_{\omega_1^{ck}+1}$-generically c.e.\ copy, but no $\Sigma_{\omega_1^{ck}+2}$-generically c.e.\ copy?
\end{question}

\section{Constructions of Individual Orderings}
It is apparent from the previous section that there ought to be some interesting computability theoretic proprieties hidden in the analysis of particular structures.
The nature of "having a copy" is structural so our arguments were fairly structural.
In this section we move away from the idea of "having a copy" with a particular computability theoretic property and explore some specific constructions.
That being said, our techniques remain fairly structural in this section.

In the first subsection we construct linear orderings that are $\Sigma_{\alpha+2}$-coarsely c.e.\ but not $\Sigma_{\alpha+2}$-generically c.e.\ for each $\alpha\in\omega_1^{ck}$.
We move on to give orderings with a non-zero c.e.\ degree $c$ that have $\Sigma_2$-generically c.e, copies, and no $\Sigma_3$-generically c.e. copy.
Finally, we investigate dense subsets of the universal linear order $\eta$ with respect to different enumerations of $\eta$.

\subsection{Coarsely c.e.\ orderings that are not generically c.e.}
Our first goal is to find a $\Sigma_2$-coarsely c.e.\ structure which is not $\Sigma_2$-generically c.e.
That being said, our previous work already indicated what this will be.

\begin{proposition}
    If $A$ has no infinite $\Sigma_3^0$ subsets, then $Sh(A)$ has a $\Sigma_2$-coarsely c.e.\ copy that is not $\Sigma_2$-generically c.e.
\end{proposition}

\begin{proof}
Consider the copy construction in Proposition \ref{ShCoarse}.
By construction this copy is $\Sigma_2$-coarsely c.e.
However, it follows from Proposition \ref{shuffle} that this copy cannot be $\Sigma_2$-generically c.e.
\end{proof}

We can now iterate this construction up the hierarchy.

\begin{proposition}
    For any $\alpha\in\omega_1^{ck}$ there is a linear ordering that is $\Sigma_{\alpha+2}$-coarsely c.e.\ but not $\Sigma_{\alpha+2}$-generically c.e.
\end{proposition}

\begin{proof}
    Let $A$ be such that every infinite subset computes Kleene's $\mathcal{O}$ and consider
    $$L_\gamma=
    \begin{cases}
    \zeta^\beta\cdot Sh(A),& \text{if } \gamma=2\beta+2\\
    \zeta^\beta\cdot(\eta+2+\eta)\cdot Sh(A),              & \text{if } \gamma=2\beta+3.
    \end{cases}$$
    By Proposition \ref{parity} and Proposition \ref{reduce} it follows that $L_\gamma$ does not have any hyperarithmetic $\Sigma_{2\beta+3}$ elementary substructures, else $Sh(A)$ would have a hyperarithmetic $\Sigma_2$ elementary substructure, a contradiction.
    This means that $L_\gamma$ has no $\Sigma_\gamma$-generically c.e.\ copy.

    We now show that $L_\gamma$ has a $\Sigma_\gamma$-coarsely c.e.\ copy.
    Let $a\in A$ be the smallest value in $A$ and let $b\in A$ be the second smallest value.
    Let
    $$K_\gamma=
    \begin{cases}
      \zeta^\beta\cdot Sh(\omega),& \text{if } \gamma=2\beta+2\\
    \zeta^\beta\cdot(\eta+2+\eta)\cdot Sh(\omega),              & \text{if } \gamma=2\beta+3.  
    \end{cases}
    $$
    Consider a computable presentation of $K_\gamma$ where the copies of $b$ are on a dense set.
    Let $D\subset K_\gamma$ correspond to all blocks of size in $A-\{a\}$.
    Note that this is a dense set in this particular copy.
    Furthermore, $D\cong  \zeta^\beta\cdot Sh(A-\{a\})$ or $D\cong  \zeta^\beta\cdot(\eta+2+\eta)\cdot Sh(A-\{a\})$ as a substructure.
    There is some presentation of $L_\gamma$ which enumerates all blocks of size in $A-\{a\}$ exactly as in $D$ and gives all leftover values to elements in blocks of size $a$.

    What is left is to show that the canonical injections
    \begin{itemize}
        \item $D\cong \zeta^\beta\cdot Sh(A-\{a\})\hookrightarrow \zeta^\beta\cdot Sh(\omega)$
        \item $D\cong \zeta^\beta\cdot Sh(A-\{a\})\hookrightarrow \zeta^\beta\cdot Sh(A)$
        \item $D\cong \zeta^\beta\cdot(\eta+2+\eta)\cdot Sh(A-\{a\})\hookrightarrow \zeta^\beta\cdot(\eta+2+\eta)\cdot Sh(\omega)$
        \item $D\cong \zeta^\beta\cdot(\eta+2+\eta)\cdot Sh(A-\{a\})\hookrightarrow \zeta^\beta\cdot(\eta+2+\eta)\cdot Sh(A)$
    \end{itemize}  are $\gamma$-embeddings.
    This follows from Lemmas \ref{lem:zpowersbfpreserve}, \ref{e2ePreserve} and \ref{Sh2Lem}.
\end{proof}

The above construction gives that there is no computable point in the $\Sigma_\beta$ hierarchy at which coarse and generic computability coincide from then on. 
Note that we leave the converse construction open.

\begin{question}
    Is there a linear ordering that is $\Sigma_{\alpha+2}$-generically c.e.\ but not $\Sigma_{\alpha+2}$-coarsely c.e.?
\end{question}

\subsection{Intersections with ordinary computability}
We now consider the way being $\Sigma_n$-generically c.e.\ intersects with more ordinary notions of computability.
Given a non-zero c.e.\ degree $c$, Jockusch and Soare \cite{JockuschSoare1991} constructed a linear ordering $L(c)$ of degree $c$ with no computable copy.
Notably, their linear ordering is always of a very particular isomorphism type such that
$$L(c)\cong \sum_{i\in\omega} S_i+A_i,$$
with $S_i\cong 2+\eta+i+2+\eta+2$ and either $A_i\cong \omega$ or $A_i\cong k+\omega^*$ for some finite $k$.
The relative simplicity of the isomorphism type plays a key role in the construction as written.
In particular, it is observed that the points in each $S_i$ that are not in one of the copies of $\eta$ are defined by a $\Pi_2^0$ predicate.
This allows the predicate to be defined by the presence of infinitely many witnesses to something computable and be effectively approximated throughout the construction.

One side effect of the relatively simple isomorphism type of these linear orderings is that they are fairly easy to analyze in terms of generic computability.

\begin{proposition}
    Given a non-zero c.e.\ degree $\mathbf{c}$, there is a linear order $L(c)$ of degree $\mathbf{c}$ with a $\Sigma_2$-generically c.e.\ copy, and no $\Sigma_3$-generically c.e.\ copy.
\end{proposition}

\begin{proof}
    Consider the construction of Jockusch and Soare. Note that $L(\mathbf{c})$ must have infinitely many $i$ with $A_i\cong k+\omega^*$ for $k\in\omega$; otherwise $L(\mathbf{c})$ has a computable copy, a contradiction. Therefore, there is an embedding $\omega\hookrightarrow(L(\mathbf{c}),\omega^*)$. The argument from Proposition \ref{scatteredComp} now gives that this guarantees that $L(\mathbf{c})$ has a $\Sigma_2$-generically c.e.\ copy.

    We now show that there is no $\Sigma_3$-generically c.e. copy for $L(\mathbf{c})$.
    In fact, we show that any $\Sigma_3$ elementary substructure of $L(\mathbf{c})$ is the whole of $L(\mathbf{c})$.
    Consider a $\Sigma_3$ elementary substructure $K$ of $L(\mathbf{c})$.
    There is a $\Sigma_3$ sentence that states that there are intervals isomorphic to each $S_i$.
    In a manner analogous to the formulas $(*)$ and $(**)$ defined in Proposition \ref{parity} we can say that there is a sequence $u<x_1<\cdots<x_i<v$ such that $x_i$ and $x_{i+1}$ are successors and the intervals $(u,x_1)$ and $(x_i,v)$ are dense.
    These intervals only appear in the orderings once, so $K$ must include all of these intervals in the specified order.
    In summary, $K\cong \sum_{i\in\omega} S_i+B_i,$ where the $B_i$ $\Sigma_3$ elementary embed into the $A_i$.
    Each of the $A_i$ has a $\Pinf{3}$ Scott sentence, so this means that each $B_i$ must in fact be isomorphic to $A_i$.
    Because the embedding must saturate blocks, $B_i=A_i$ is also forced.
    In total, this means that if $K$ is a witness to $\Sigma_3$-generically computability, it also demonstrates that $L(\mathbf{c})$ was computable to begin with, a contradiction
\end{proof}

Note that the above proposition is equally true if we consider only first order formulas.
This follows from the fact that the formulas $(*)$ and $(**)$ from Proposition \ref{scatteredComp} are first order along with the observation that any $\exists_3$ substructure of $\omega$ or $k+\omega^*$ must in fact be all of $\omega$ or $k+\omega^*$ respectively.

When this is propagated up the hierarchy in a naive manner we can produce some more interesting classes of examples.
That being said, they introduce some serious limitations on the degree of the linear order.

\begin{theorem}
     Given a non-zero, non-$n$-low degree $\mathbf{c}\leq \mathbf{0}'$ there is a linear ordering $L(\mathbf{c})$ of degree $\mathbf{c}$ with a $\Sigma_{n+2}$-generically c.e.\ copy, and no $\Sigma_{n+3}$-generically c.e.\ copy.
\end{theorem}

\begin{proof}
    (sketch) By assumption $\mathbf{0}^n<_Tc^n\leq_T \mathbf{0}^{n+1}$.
    By Sacks' theorem \cite{Sacks63} $\mathbf{c}^n$ is c.e.\ in $\mathbf{0}^n$.
    If we relativize the construction of Jockusch and Soare \cite{JockuschSoare1991} we get a linear order $L_n(\mathbf{c})$ that has degree $c^n$ but does not have a $\mathbf{0}^n$-computable copy.
    For $m=\lfloor \frac{n}2\rfloor$, we let $L(\mathbf{c})=\zeta^m\cdot L_n(\mathbf{c})$ if $n$ is even and $L(\mathbf{c})=(\eta+2+\eta)\cdot\zeta^m\cdot L_n(\mathbf{c})$ if $n$ is odd.
    By the theorems of Fellner \cite{Fellner} and Downey and Knight \cite{DK92}, observe that $L(\mathbf{c})$ has a $\mathbf{c}$-computable copy but no computable copy.
    Finally, Lemma \ref{lem:zpowersbfpreserve} and Proposition \ref{parity} allows us to modify the previous proof to show that $L(\mathbf{c})$ has a $\Sigma_{n+2}$-generically c.e.\ copy, and that any $\Sigma_{n+3}$ elementary substructure must be all of $L(\mathbf{c})$.
    This completes the construction and proof.
\end{proof}

An interesting question that naturally emerges from the above observations is if we can get rid of the above restrictions on the degree $\mathbf{c}$.

\begin{question}
    For $n\geq 2$, given a non-zero c.e.\ degree $\mathbf{c}$ is there a linear ordering $L(\mathbf{c})$ of degree $\mathbf{c}$ with a $\Sigma_n$-generically c.e.\ copy, and no $\Sigma_{n+1}$-generically c.e.\ copy?
\end{question}

\subsection{Notions of Density}
We now investigate subsets of the universal linear ordering $\eta$ under different enumerations of this ordering and connections with coarse computability.

\begin{definition} Let $\{\phi(i): i < \omega\}$ be a computable one-to-one enumeration of the set $\bQ$ of rational numbers (or of $\omega$).
\begin{enumerate}
\item
The notion of density $\delta_{\phi}$  for subsets of $\bQ$ determined by $\phi$  is 
\[
\delta_{\phi}(S) = \lim_{n \to \infty} \frac {|S \cap \{\phi(0),\dots,\phi(n-1)\}|} n.
\]
\item A notion of density $\delta = \delta_{\phi}$ is said to be \emph{positive} if for each nonempty interval $(a,b)$ of $\bQ$, 
$\delta(a,b) > 0$.  To be more precise, if  $\liminf_{n \to \infty} \frac {|(a,b) \cap \{\phi(0),\dots,\phi(n-1)\}|} n > 0$. 
\end{enumerate}
\end{definition}

This can be extended to the usual definition of density as a limsup. 

Here are two examples, using the set $B$ of dyadic rationals in $(0,1)$ for simplicity. 

\begin{example}
\begin{enumerate}
\item[(1)] Let $\sigma_1 = 1, \sigma_2 = 01, \sigma_3 = 11, \sigma_4 = 001, \sigma_5 = 101, \dots$ enumerate the (reverse) binary representations of the positive natural numbers and let \[\phi(n) = \sigma_n = \sum \{2^{-i-1}:  \sigma_n(i) =1\}.\] Then for any dyadic interval $(p,q)$, we have $\delta_{\phi}(p,q) = q-p$. 

Consider $(\frac14,\frac38)$.  Dyadic rationals in this interval will be those starting with (010), which corresponds to natural numbers congruent to 2 modulo 8. 
So the interval will contain exactly 1 out of every 8 consecutive numbers in this enumeration. It follows that the interval has $\delta_{\phi}$ density equal to $\frac18$, as expected. 

\smallskip

\item[(2)] Next enumerate the dyadic rationals in $[0,1]$ first by denominator and next by numerical order, so that the enumeration is $0, 1, \frac12, \frac14, \frac34, \frac 18, \frac38, \dots$. (In the reverse binary notation above, this will be ordered first by length and then 
lexicographically.)  Then for any interval $S$, we will have 
\[\liminf \frac {|S \cap \{\phi(0),\dots,\phi(n-1)\}|} n < \limsup   \frac {|S \cap \{\phi(0),\dots,\phi(n-1)\}|} n.\]
 
 Consider in particular the interval $(\frac12,1)$.  This will contain exactly one half of the dyadic rationals enumerated up to a given length $n$, that is, $2^{n-1}$ out of $2^n$. 
But then the next $2^n$ numbers will all fall into $(0,\frac12)$, so the density of the interval after that will be $2^{n-1}$ out of $2^{n+1}$, that is, $\frac14$. 
So the $\liminf$ will be $\frac14$ and the $\limsup$ will be $\frac12$.  This will still count as a positive notion of density. 
\end{enumerate}
\end{example}

More generally, if $\phi$ enumerates the rationals $\pm p/q$ (with $p,q$ relatively prime) in order first by the sum $p+q$ and then by the numerator $p$, it may be seen that $\delta_{\phi}$ is a positive notion of density. 

Since every countable (respectively, computable) linear ordering is (resp. computably) isomorphic to a (respectively, computable) substructure of $\bQ$ (with the standard order), we can make the following connection.

\begin{lemma} 
Let  $\sA = (A,\leq^A)$ be a coarsely computable linear ordering.
Let $\E$ be the computable linear ordering and $D$ be the corresponding dense set so that $\sD$ is a substructure of both $\sA$ and $\E$. Then there is a notion of density $\delta$ and a $\delta$-dense computable subset $P$ of $\bQ$ such that $(P,\leq)$ is isomorphic to $\sD$. 
\end{lemma}

\begin{proof}  Let $\sB$ be a computable substructure of $\bQ$ which is computably isomorphic to $\E = (\omega,\leq^E)$ via the 
map $\psi: E \to B$ and let $P = \psi[D]$. Since $D$ is dense in $E$, we have
\[
\lim_{n \to \infty} \frac{|D \cap n|} n = 1.
\]
Choose an enumeration $\theta$ of $\bQ$ so that 
\[
\lim_{n \to \infty} \frac{|\{\psi(0),\psi(1),\dots,\psi(n-1)\} \cap \{\theta(0),\dots,\theta(n-1)\}|} n = 1.
\]

For example, we could define $\theta$ so that, for each $n$, the set $\{\theta(0),\dots,\theta(2^n-1)\}$ includes $\{\psi(0),\dots,\psi(2^n-n)\}$ (in order) along with the first $n$ elements of $\omega \setminus B$. 

It follows that 
\[
\lim_{n \to \infty} \frac{|P \cap \{\theta(0),\dots,\theta(n-1)\}|} n = 1,
\]
so that $P$ is dense in $\bQ$. 
\end{proof}

\medskip

Thus we are led to consider the following problem: 

\medskip

Problem:  For a given notion of density, for which countable linear orderings $L$ does there exist a dense subset $\sD$ of $\bQ$ with order type $L$?

\begin{proposition} If $\delta$ is a positive notion of density, then any $\delta$-dense substructure $\sD$ of $\bQ$ has no successors. This means that $\sD$ has order type $\eta$, $1+\eta$ or $\eta +1$.  
\end{proposition}

\begin{proof} If $p < q$ are both in $D$ but $(p,q) \cap D = \emptyset$, then $D$ is missing an interval of positive density, so $D$ cannot be dense. 
\end{proof}

\begin{proposition}
    For any fixed (computable) notion $\delta$ of density on $\bQ$, there is a computable linear ordering which is not isomorphic to a $\delta$-dense subset of $\bQ$. 
\end{proposition}

\begin{proof} Suppose that $\omega$ is isomorphic to a dense subset $D$ of $\bQ$. It follows that any dense subset of $\bQ$ must include an infinite subset of $D$ and thus a copy of $\omega$.  Hence there can be no dense subset of $\bQ$ with order type $\omega^*$.  
\end{proof} 

\begin{proposition} Let $L =  (\omega,\leq^L)$ be a countable linear ordering.  
Then there exists a computable notion of density $\delta$ and a copy $L'$ of $L$ such that $L'$ is coarsely computable with respect to 
$\delta$ if and only if  $L$ has an infinite subset $D$ such that $(D,\leq^L)$ is isomorphic to a computable ordering.  
\end{proposition} 

\begin{proof} The first direction is immediate from the definition. 

Now suppose that $L$ has a subset  $D$ such that $\sD = (D,\leq^L)$ is isomorphic to a computable linear ordering and hence isomorphic to a computable subset $P$ of $\bQ$. 
Let $G: \sD \to P$ be an isomorphism. Let $\delta$ be any computable notion of density such that $P$ is $\delta$-dense in $\bQ$.  Extend $G$ to a set isomorphism from $\omega$ to $\bQ$ and 
define $L' = (\bQ,\leq')$  by letting $x \leq' y \iff G^{-1}(x) \leq G^{-1}(y)$. Then $L'$ is isomorphic to $L$ and $P$ is a dense subset of both. 
\end{proof}

\section{Beyond Linear orderings}
In this section we present some results outside of the confines of linear orderings.
Note that some of our previous results arguably fall in this purview as well.
For example, the $\mathbf{\Sigma}_1^1$ hardness results also tell us that it is $\mathbf{\Sigma}_1^1$-hard to say that a structure in general has a $\Sigma_{\alpha+2}$-generically c.e.\ or $\Sigma_{\alpha+2}$-coarsely c.e.\ copy, which represents a new result even in the more general setting.
We produce here an example of a set of structures that sees the $\mathbf{\Sigma}_1^1$ hardness result at the level $\Sigma_1$.
We will define the $\alpha$-Ramsey property which reframes some of our study in terms of Ramsey theoretic considerations and is used to show a general result about  generic and coarse computability.

\subsection{Hardness at the level of $\Sigma_1$}
We now prove the claim that not all relational structures have a $\Sigma_1$-generically c.e.\ copy.
In fact we do not need to move very far from linear orderings to find such structures.

\begin{definition}
A successor linear ordering is a triple $\L=(L,\leq,S)$ where $S$ is a two place relation that hold exactly of successors.
\end{definition}

\begin{theorem}
     The set of successor linear orderings with a $\Sigma_1$-generically c.e.\ copy is $\mathbf{\Sigma}_1^1$ complete.  
\end{theorem}

\begin{proof}
    Given a linear ordering $L$, take the structure $(Sh(A)+\omega)\cdot L$ as a successor linear ordering for a fixed $\Sigma_2$ immune $A$.
    If $L$ is ill-founded, then $(Sh(A)+\omega)\cdot L$ has a 1-embedding of $\omega\cdot\omega^*$.
    This follows from the fact that the embedding as plain linear orderings is a 2-embedding as seen in Proposition \ref{scatteredComp} and that $S$ is definable in one quantifier in the language of linear orderings.
    As $\omega\cdot\omega^*$ has a computable copy as a successor linear ordering, this means that $(Sh(A)+\omega)\cdot L$ has a $\Sigma_1$-generically c.e.\ copy by Proposition \ref{genChar}.
    
    If $L$ is well founded, note that any 1-embedding must saturate 1-blocks, preserve the $\text{Succ}_n$ relations and preserve the existence of least elements.
    Consider a structure $N$ that 1-embeds into $(Sh(A)+\omega)\cdot L$.
    Let $k\in L$ be least such that there is some element $(x,k)\in N$.
    Consider the set of elements $x\in (Sh(A)+\omega)$ such that $(x,k)\in N$ and call this set $S$.
    Note that because $(Sh(A)+\omega)\cdot L$ has no least element, $N$ must also have no least element.
    In particular $S$ must have no least element.
    Note that if $S\subset\omega$ it will have a least element, therefore $S\cap Sh(A)\neq\emptyset$.
    In fact, $S\cap Sh(A)$ cannot be finite, otherwise $S$ would still have a least element.
    As every block is finite, this means that $S$ contains infinitely many blocks of $Sh(A)$.
    Between any two blocks of $Sh(A)$ in $N$, $N$ must contain arbitrarily long successor chains.
    This means that $Bk(N)$ is an infinite subset of $A$ and is therefore not $\Sigma^0_2$.
    However, if $N$ were isomorphic to a c.e.\ structure, $Bk(N)$ would be $\Sigma^0_2$ as the definition of $Bk(N)$ is $\Sigma^0_2$ in the successor relation.
    This means that $(Sh(A)+\omega)\cdot L$ has no $\Sigma_1$ elementary substructure isomorphic to a c.e.\ structure.
    In particular, it is not $\Sigma_1$-generically c.e.

    The above reduction shows that the set of successor linear orderings with a $\Sigma_1$-generically c.e.\ copy is $\mathbf{\Sigma}_1^1$ hard.
    To show it is complete we must note that it is itself $\mathbf{\Sigma}_1^1$.
    This follows from writing out the ordinary definition of the concept.
    $L$ has a $\Sigma_1$-generically c.e.\ copy if and only if,
    $$\exists M\cong L ~~ \exists S\subseteq M ~~ \text{S is dense, c.e.\ and } \forall \bar{p}\in S (S,\bar{p})\equiv_1(M,\bar{p}),$$
    which is $\mathbf{\Sigma}_1^1$ as being dense and c.e.\ and saying $(S,\bar{p})\equiv_1(M,\bar{p})$ are all arithmetic.
\end{proof}

The techniques from the previous sections can extend this argument to the analysis of having a $\Sigma_1$-coarsely c.e.\ copy.
These proof is omitted here, as it is quite similar to those that were already produced.

\begin{theorem}
     The set of successor linear orderings with a $\Sigma_1$-coarsely c.e.\ copy is $\mathbf{\Sigma}_1^1$ complete.  
\end{theorem}

\subsection{The $\alpha$-Ramsey property}
We take a global view of the facts that we used to demonstrate the existence of $\Sigma_1$-generically c.e.\ copies of linear orderings and of $\Sigma_2$-generically c.e.\ copies of scattered linear orderings (along with the coarse analogues).
The global property that we observe bears a resemblance to Ramsey's theorem.
The original Ramsey's theorem states that any infinite graph has an infinite clique or independent set as an induced subgraph.
The overall form of this theorem is ``for any structure in a class $\mathbb{K}$ it has a well-behaved substructure.''
That being said, the induced substructure relation is fairly weak and is arguably overly weak.
In particular, the presence of an infinite clique or independent set tells us nothing about the structure of the original graph.
The natural strengthening of this, especially considering the work done with generically computable and coarsely computable structures, is to demand that our substructures pick up more structural information.
In particular, we may ask them to be a $\Sigma_\alpha$ substructure instead of just any substructure.
In doing this, we are asking for a lot more, so we may have to relax our definition of ``well-behaved''.
One choice that can be made is to identify ``well-behaved'' with computable.
This is exactly the idea of the following definition.

\begin{definition}
A class of structures $\mathbb{K}$ has the \textit{$\alpha$-Ramsey property} if every $\M\in\mathbb{K}$ has an infinite $\Sigma_\alpha$ substructure with a computable copy.
We say the a set of computable structures $\mathbb{S}$ \textit{witnesses} the $\alpha$-Ramsey property for $\mathbb{K}$ if every $\M\in\mathbb{K}$ has an infinite $\Sigma_\alpha$ substructure isomorphic to a structure in $\mathbb{S}$.
\end{definition}

Note that taking $\alpha=0$ reduces our demand down to having any computable substructure in a situation analogous to the original Ramsey theorem.
To be explicit, Ramsey's theorem implies that graphs have the $0$-Ramsey property and this is witnessed by the set containing the empty graph and the complete graph.
As we have already observed, the $\alpha$-Ramsey property has implications for generic computability and coarse computability.

\begin{lemma}
If a class of structures $\mathbb{K}$ has the $\alpha>0$ Ramsey property, every $\M\in\mathbb{K}$ has a $\Sigma_n$-generically c.e.\ copy.
If a class of structures $\mathbb{K}$ has the $0$-Ramsey property, every $\M\in\mathbb{K}$ has a generically computable copy.
\end{lemma}

\begin{proof}
See Proposition \ref{genChar}.
\end{proof}

To find the right statement for coarsely computable structures, we must first introduce a new notion.

\begin{definition}
For structures $\M$ and $\N$, we say that $\N$ $\alpha$-covers $\M$ if $\M$ $\alpha$-embeds into $\N$ with a computable image.
We say that a class of structures $\mathbb{K}$ $\alpha$-covers a class of structures $\mathbb{J}$ if every structure $\sB\in\mathbb{J}$ and every $\beta\in\omega+1$, $\sA$ is $\alpha$-covered by a structure $\sA\in\mathbb{K}$ such that $\vert \sA-\sB\vert=\beta$.
\end{definition}

As matter of notation, given a class $\mathbb{K}$ of structures, we let $\mathbb{K}_c$ be the class of computable copies of models in $\mathbb{K}$.

\begin{lemma}
If a class of structures $\mathbb{K}$ has the $\alpha>0$ Ramsey property witnessed by $\mathbb{S}$ and $\mathbb{K}_c$ $\alpha$-covers $\mathbb{S}_c$, then every $\M\in\mathbb{K}$ has a $\Sigma_n$-coarsely c.e.\ copy.
Similarly, if a class of structures $\mathbb{K}$ has the $0$-Ramsey property witnessed by $\mathbb{S}$ and $\mathbb{K}_c$ $0$-covers $\mathbb{S}_c$, every $\M\in\mathbb{K}$ has a coarsely computable copy.
\end{lemma}

\begin{proof}
See Lemma \ref{coarse.struct}.
\end{proof}

Note that the proofs for linear orderings actually dealt directly with the $\alpha$-Ramsey property.
In particular, we immediately have the following from our previous work,

\begin{corollary}
\begin{enumerate}
    \item The class of linear orderings has the $1$-Ramsey property witnessed by 
    $$\{\omega, ~\omega^*, ~\zeta,~\omega\cdot\omega^*, ~\omega^*\cdot\omega, ~\eta\}.$$
    \item The class of scattered linear orderings has the $2$-Ramsey property witnessed by  
    $$\{\{\omega+k\}_{k\in\omega}, ~\omega+\omega^*, ~\{k+\omega^*\}_{k\in\omega}, 
    ~\{k+\zeta+k'\}_{k,k'\in\omega}, ~\{\{k+\omega\cdot\omega^*+k'\}_{k,k'\in\omega}, 
    ~\{k+\omega^*\cdot\omega + k'\}_{k,k'\in\omega} \}.$$
\end{enumerate}

\end{corollary}

The necessary covering property to give coarsely computable copies is not too difficult to see in this case as observed in Theorem \ref{scatteredThm}.

\begin{corollary}
    The class of structures 
    $$\{\{\omega+k\}_{k\in\omega}, ~\omega+\omega^*, ~\{k+\omega^*\}_{k\in\omega},
    ~\{k+\zeta+k'\}_{k,k'\in\omega}, ~\{\{k+\omega\cdot\omega^*+k'\}_{k,k'\in\omega}, 
    ~\{k+\omega^*\cdot\omega + k'\}_{k,k'\in\omega} \}$$ is $\alpha$-covered by the class of computable linear orderings.
\end{corollary}

A strong implication of this framework is the following general theorem mentioned in the introduction.

\begin{theorem}
The set of infinite models of a theory over a finite relational language has the $0$-Ramsey property and therefore each model has a generically computable copy. Furthermore, every such model has a coarsely computable copy.
\end{theorem}

\begin{proof}
Say the language is $R_1,\dots, R_n$.
Let the maximum arity of the $R_i$ be $k$.
Fix an infinite model $\M$ and color each tuple of length $k$ by its quantifier free type.
Note that there is some finite number of possible quantifier free types.
Pick an infinite monochromatic subset $T\subseteq M$.
This subset is a submodel because the language is relational.
Furthermore, the restriction of $\M$ to $T$ is computable because the relations on any tuple are fixed by the single color that monochromatically colors $T$.
In particular, $T$ has a computable copy given by fixing all relations according to the selected color.

To see that $\M$ also has a coarsely computable model, note that $T$ can be computably embedded into a copy of itself.
In fact, any infinite subset of $T$ is a copy of $T$ by construction.
\end{proof}

Note that this, in particular, implies the result on coarsely and generically computable copies of equivalence relations seen in \cite{CCH22}.

\bibliographystyle{amsplain}

\end{document}